\documentclass[8pt]{article}
\usepackage[utf8]{inputenc}
\usepackage[T1]{fontenc}
\usepackage[numbers]{natbib}
\usepackage{graphicx}
\usepackage{amsmath}
\usepackage{amssymb}
\usepackage{bm}
\usepackage{hyperref}
\usepackage{amsfonts}
\usepackage{mathtools}
\usepackage{enumitem}
\usepackage{tikz}
\usepackage[a4paper, margin=1in, footskip=0.25in]{geometry}
\usepackage{amsthm}
\usepackage{xcolor}
\usepackage{stmaryrd}
\newtheorem{theorem}{Theorem}[section]

\newtheorem{lemma}[theorem]{Lemma}
\newtheorem{proposition}{Proposition}[section]

\theoremstyle{definition}
\newtheorem{definition}{Definition}[section]

\newtheorem{remark}{Remark}[section]

\newcommand{\wi}{W^{H,\perp}}

\newcommand{\NN}{\mathbb{N}}

\newcommand{\EE}{\mathbb{E}}
\newcommand{\PP}{\mathbb{P}}

\newcommand{\FF}{\mathcal{F}}

\newcommand{\RR}{\mathbb{R}}

\newcommand{\absv}[1]{\left| #1 \right|}

\newcommand{\inv}[1]{\frac{1}{#1}}

\newcommand{\norm}[1]{\left\Vert #1 \right\Vert}

\newcommand{\indic}[1]{\mathbf{1}_{[#1]}}

\newcommand{\pdx}{\partial_x}

\newcommand{\pdt}{\partial_t}

\newcommand{\wh}{W^H}

\newcommand{\ls}{\lesssim}
\newcommand{\eps}{\varepsilon}

\newcommand{\sgn}[1]{\text{sgn}\left( #1 \right)}

\newcommand{\ccc}[1]{\mathcal{C}^{#1}}

\newcommand{\blue}[1]{\textcolor{black}{#1}}

\title{Zero noise limit for singular ODE regularized by fractional noise}

\author{Paul Gassiat$^1$, Łukasz Mądry$^2$}
\date{}
\begin{document}

%

\maketitle 

\begin{abstract}
We consider a scalar ODE with a power singularity at the origin, regularized by an additive fractional noise. We show that, as the intensity in front of the noise goes to $0$, the solution converges to the extremal solutions to the ODE (which exit the origin instantly), and we quantify this convergence with subexponential probability estimates. This extends classical results of Bafico and Baldi in the Brownian case. The main difficulty lies in the absence of the Markov property for the system. Our methods combine a dynamical approach due to Delarue and Flandoli, with techniques from the large time analysis of fractional SDE (due in particular to Panloup and Richard).
\end{abstract}


%


\section{Introduction}

The purpose of this work is to investigate the so-called zero noise limit of a stochastic differential equation of the form :

\begin{equation}\label{eq:intro_equation}
X_t^{\eps} = X_0 + \int_0^t f(X^{\eps}_r) dr + \eps W_t.
\end{equation}

We are interested in the case when $f$ is singular (in particular, not Lipschitz continuous), and the deterministic equation $\dot{x}= { f(x) }$ is ill-posed, but the additive noise term $\eps W$ has a regularizing effect so that, for each $\eps>0$, \eqref{eq:intro_equation} has a unique solution. This phenomenon is typically \blue{referred} to as ''regularization by noise'' for ODE\blue{s}. In the case where $W$ is a Brownian motion, this goes back to old works of \cite{Zvo74}, \cite{Ver81} who showed that $f$ bounded is sufficient for well-posedness of \eqref{eq:intro_equation}.
\medskip

Heuristically, regularization by noise only requires the path $W$ to be sufficiently irregular (or, equivalently, its occupation measure to be sufficiently regular) to counter-act irregularity of the drift $f$, so that in principle one could allow $W$ to be taken in a much wider class than classical Brownian motion. In this context, fractional Brownian motion (fBm) $W^H$ is a natural candidate because, in particular, it comes with a parameter $H \in (0,1)$ which allows to tune its regularity. Early results in the scalar case were obtained in the work of Nualart-Ouknine \cite{NO02}, but the breakthrough came with the work of Catellier-Gubinelli \cite{CG16} who showed that \eqref{eq:intro_equation} is well-posed as long as
\[
b \in C^\gamma, \;\;\; \gamma > 1 - \frac{1}{2H}. 
\]
Note that $\gamma$ is allowed to be negative (in which case the integral $\int_0^{\cdot} b(X^\eps_s) ds$ needs to be suitably interpreted), and in fact, a particularly striking feature of the above result is that, as $H \to 0$, the regularity requirement covers the whole range of $\gamma \in (-\infty,1)$.
\medskip

Following \cite{CG16}, regularization of ODE by fractional noise has been a very active area of research, in particular with the help of tools based around L\^e's stochastic sewing lemma \cite{Le20}. Let us mention for instance the study in this context of : McKean-Vlasov equations  \cite{GHM22},  $C^\infty$-regularization \cite{HP21},  convergence rates of numerical schemes \cite{BDG19,GHR23}, time-dependent drifts \cite{GG22}, equations with multiplicative noise \cite{catellier_duboscq_rde_25, DG22, GH22, BH23, MM23}. 
\medskip

Once one has a well-posedness result for \eqref{eq:intro_equation}, it is then a natural question whether, in the limit as $\eps \to 0$, the solutions $X^\eps$ converge to one (or several) specific solution(s) to the deterministic equation. These solutions could then be interpreted as the correct ''physical'' solution to the original equation, a phenomenon which has been called ''selection by noise''  {(see for instance the discussion in Section 1.4 of \cite{Fla11})}.
\medskip

In the case where $W$ is a Brownian motion, and the equation is scalar (with an isolated singularity), this question has been solved by Bafico-Baldi \cite{BB82} who showed that, as $\eps \to 0$, the solutions $X^\eps$ concentrate near the ''extremal'' solutions which exit the singularity immediately. Their methods relied on explicit solutions to elliptic PDE, and as such were limited to the one-dimensional case. Still in the scalar case, Gradinaru-Herrmann-Roynette \cite{GHR01} showed that the large deviation rate for these equations may be different from the classical one (and depends on $f$). More recently, some approaches based on the dynamics of the solutions have been considered by Trevisan \cite{Tre13} and especially \cite{DF14}. These methods have then be applied in (among others) extensions to multiple dimensions \cite{DM19},  {Vlasov-Poisson point charge dynamics \cite{DFV14} } and mean-field games \cite{DT20}. A similar approach has also been put forward by Pilipenko and Proske \cite{PP18}, with further extensions to higher dimensions \cite{PP20} and L\'evy noise \cite{KP21}.

\medskip

Our goal in this paper is to consider the zero-noise limit of \eqref{eq:intro_equation} when $W=W^H$ is a fBm. We will focus in this paper on the simplest case, where the equation is scalar with an isolated singularity, and more precisely with a power-law drift, of the form
\[
 { f(x) } = \sgn{x}  |x|^{\gamma},
\]
for some $1 - \frac{1}{2H}  < \gamma < 1$ (suitably interpreted as Schwartz distribution on $\RR$ when $\gamma \leq -1$, \blue{see Remark \ref{rem:distribution_integer_gam} below for more details}).
\medskip

Recall, that (at least in the case $\gamma \geq 0$), the ODE $\dot{x} =  { f(x) }$ admits an infinite family of solutions $(x^{+,t_0},  x^{-,t_0})_{t_0\geq 0}$, given by
\[ t \mapsto x^{\pm,t_0}_t:= \pm c_{\gamma}\max( t - t_0, 0)^{\inv{1-\gamma}},  { \quad c_{\gamma} = (1-\gamma)^{\inv{1-\gamma}} .} \]
\medskip

Our main result states that, in this case as well, the solutions $X^\eps$ concentrate near the extremal solutions $x^{\pm,0}$ of the ODE, at times of order
\[ t_\eps := \eps^{\left( \frac{1}{1-\gamma} - H\right)^{-1}}, \]
with high (sub-exponential) probability. 
\begin{theorem}\label{thm:main_simplified}
Let $X^{\eps}$ be as in \eqref{eq:intro_equation} with $W$ a fBm of parameter  {$H \in (0,1)$}. Then there exists a random time $\psi_{\eps}$ such that, for suitable $\nu, \alpha>0$, 
either
\[ \forall t\geq 0,    X^{\eps}_{\psi_{\eps}+t}   \geq x^{+,0}_t -  \eps^{\nu} t^\alpha, \]
or
\[ \forall t\geq 0,    X^{\eps}_{\psi_{\eps}+t}   \leq x^{-,0}_t +   \eps^{\nu} t^\alpha, \]

In addition, for some $\kappa >0$, it holds, uniformly in $\eps>0$,
\begin{equation} \label{eq:SubexpIntro}
\ln \PP\left( \frac{\psi_{\eps}}{t_\eps} > z \right) \ls - z^{\kappa}.
\end{equation} 
%
\end{theorem}

We refer to Theorem \ref{thm:main} below in the main text for a more detailed and precise version which, in particular, does not require the drift to be symmetric and specifies the possible values for the parameters $\nu, \alpha, \kappa$. The theorem above provides a one-sided bound away from zero, the convergence of $X^\eps$ to $x^{\pm,0}$ requires an additional bound on $|X^\eps|$, which is obtained in Proposition \ref{prop:upper_bound} below (since this concerns the behaviour away from the singularity, this poses no difficulty).

\vspace{5mm}

\textbf{Difficulties and outline of the proof}

We first remark that the methods of \cite{BB82}, relying on PDE associated to Markov diffusions, as well as those of \cite{Tre13}, \cite{GHR01}, which rely on It\^o calculus and martingale arguments, can clearly not be applied here. 
 
The only other small-noise result for fractional noise that we are aware of is due to Pilipenko-Proske \cite{PP18}. They prove convergence to extremal solutions in the case $\gamma>0$ above. Their simple and elegant proof uses the ergodic theorem and relies on the fact that, in this case, the flow of the ODE is strictly expanding (since $f'>0$). The case $\gamma  \leq 0$, where the drift is not globally monotone, cannot be treated by their arguments (recall that the ability to take $\gamma$ arbitrarily negative is an important feature of regularization by fractional noise). In addition, their approach does not lead to any quantitative bounds such as \eqref{eq:SubexpIntro} above.

In our proof, we follow the dynamic approach put forward in \cite{DF14}. The main idea of their work can be summarized as follows. Recall that, for $W$ fBm with Hurst index $H$,
\[ \eps  W_t \approx \eps t^H, \mbox{ while }\;\;\;|x^{\pm}_t| \approx t^{\frac{1}{1-\gamma}}.\]
Note that 
\[
t_{\eps}=\eps^{\left( \frac{1}{1-\gamma} - H\right)^{-1}}
\]
is the unique $t$ such that these two terms are equal, and let
\[ 
x_\eps = \eps t_{\eps}^H = t_{\eps}^{\frac{1}{1-\gamma}} 
\]
their common value at that time.

 Delarue-Flandoli suggest (in the case $H=\frac 1 2 $) the following heuristics : 
\begin{itemize}
\item for times $t \ll t_\eps$, the system essentially behaves as $\eps W_t$ (the drift is negligible) ,
\item for times $t \gg t_\eps$, the noise is negligible and the system behaves as  {$x^{\pm}_t$},
\end{itemize}
and they call $( t_\eps, x_\eps)$ the \emph{transition point} between these two regimes, which is the scale at which the system is going to select $x^+$ or $x^-$. Their proof then relies on the following iterations 
\begin{enumerate}
\item wait until the solution hits the boundaries $\pm x_\eps$ of the critical strip, which happens in a time of order $t_\eps$,
\item wait until the solution drops down below (a multiple of) the extremal solution $x^{\pm}$.
\end{enumerate}
With positive probability, the second step never concludes, which means that the solution $X^\eps$ never goes back to $0$ and must stay close to $x^\pm$. By the Markov property, all the subsequent attempts are independent, so that, a.s., after finitely many tries, the second step is successful.

\medskip
We aim to adapt their proof to our context. However, while we can follow the outline of proof above, a major difficulty comes from the fact that we do not have the Markov property, and subsequent steps are no longer independent. In order to solve this problem, we then take inspiration from works on long-time behaviour of fractional SDE (indeed, after rescaling time by a factor $t_\eps$, the zero-noise limit over a finite interval is expressed as a problem of long-time behaviour). Starting with \cite{Hai05}, a technique to deal with the lack of Markov property has been to introduce ''waiting times'' in the proof. 

The idea is as follows : recall that, for $W=W^H$ a fBm with its natural filtration $(\mathcal{F}_u)_{u \geq 0}$, for given $s \leq u \leq v$, one can always decompose the increment $W_{u,v} = W_v - W_u$ as
\[ \ \;\;\wh_{u,v} = W^{\leq s}_{u,v}  +  W^{>s}_{u,v} , \]
where $W^{>s}$ is independent from $\mathcal{F}_s$ and  { $W^{\leq s}$ is $\FF_s$-measurable}, with the additional property that $W^{\leq s}_{u,v} $ is small if $(u-s) \gg (v-u)$. This means that, if at time $s$, the influence of the past process $W^{\leq s}$ is too large (and, for instance, goes in the wrong direction), we should simply wait long enough and it will become negligible. On a more precise level, we add at each step an ''admissibility'' condition (depending on the noise up to that time), and our scheme of proof can now be written as the succession of three stages :

\begin{itemize}
\item \textbf{Stage 1 : waiting stage}. We wait some amount of time (long enough to make the influence of the past small enough), and this step succeeds if the system is admissible at the end of it.
\item \textbf{Stage 2 : exit of the critical strip}. If Step 1 succeeded, we wait an additional amount of time. This step succeeds if, at the end of it, $X^\eps$ is outside of some interval of the form $[-C x_\eps, C x_\eps]$,  and the system is still in an admissible stage.
\item \textbf{Stage 3 : staying above the extremal solution} If Step 2 succeeded, we wait until $X^\eps$ goes back below (a multiple) of $x^\pm$. This step succeeds if this never happens.
\end{itemize}

On a technical level, the way we define the waiting times in Stage 1 is strongly inspired by the recent work of Panloup-Richard \cite{PR20} (i.e. the waiting step after failing $k$ times is given as a well-chosen power of $k$ and of the norm of the noise in the most recent interval).

The main technical work is then to ensure that, given the way we have defined the various times in the stages above, each stage always has a (conditional) probability of succeeding that is (uniformly)  strictly positive, which then allows us to obtain quantitative bounds on the time after which the solution stays above (or below) the extremal curve.

\vspace{5mm}
\textbf{Outline of the paper.} The rest of the paper is structured as follows. First, we conclude this introduction by introducing some notation. In Section \ref{sec:result}, we then  state and comment on our main result. In Section \ref{sec:main_section}, we first define precisely the scheme of proof and the associated random times, and in Subsections \ref{subsec:stage_one}, \ref{subsec:stage_two} and \ref{subsec:stage_three}, we describe precisely our estimates on the three subsequent stages of the dynamics. In Section \ref{sec:proof_main} we show how these elements of the construction come together and give suitable tail estimates. Some technical estimates are relegated to Section \ref{sec:technical}.

\subsection{Notation}

We will employ standard notation - for any path $(u_t)_{t \in [0,T]}$ where $T$ is fixed, H\"older seminorm is denoted with:

\[ \norm{ u }_{\ccc{\alpha};[0,T]} = \sup_{s<t<T} \frac{ \absv{u_t - u_s } }{ \absv{t-s}^{\alpha} }, \]

and we will say that $u \in \ccc{\alpha}$ if $\norm{ u }_{\ccc{\alpha}} < \infty$. Analogously, we will denote the supremum norm with $\norm{ w}_{\infty}$, i.e.:

\[ \norm{ w }_{\infty;[0,T]} = \sup_{r <T} \absv{ w_r } .\]

In both cases we will drop the $[0,T]$ from notation if it is clear from context. 

We will also use the following notation to denote equalities and inequalities in law - let $X,Y$ be two random variables. If $X$ and $Y$ are equal in law, we will denote it as $X =_{\PP} Y$. Similarly, we say that $X$ is stochastically dominated by $Y$ if:

\[ \forall z \in \RR\;\;\PP\left( X > z \right) \leq \PP\left( Y > z \right)  .\]

We will denote it with $X \leq_{\PP} Y$. We will use standard notation to denote integration (expectation), i.e. expected value of $X$ is denoted with $\EE X$. We will write $(\FF_t)_{t \geq 0}$ to denote a filtration. For some fixed $\tau > 0$ notation $\EE^{\FF_{\tau}} \left[ X \right]$ will signify the conditional expectation and $\PP^{\FF_{\tau}}\left( A \right) = \EE^{\FF_{\tau}} \indic{A}$ for an event $A$. If a random variable $Y$ is $\FF_{\tau}$-measurable, we will denote it with $Y \in \FF_{\tau}$.

Finally, we will write $a \lesssim b$ when $a \leq C b$ for some constant $C$. If the constant $C$ depends on some real-valued parameters $u,v,g$ (for instance), then we may write $a \ls_{u,v,g} b$. Analogously, we will write $a \simeq b$ if $a \ls b$ and $b \ls a$. Similarly, when need arises to compare two random variables $X,Y$, we will write $X \ls_{\PP} Y$. 

\section{Main result} \label{sec:result}

Throughout the rest of the paper, $f$ will be fixed as an element in $C^{\gamma}(\RR)$ (i.e. the Besov-H\"older space $B^{\gamma}_{\infty,\infty}(\RR)$) for some $\gamma \in \RR$, s.t. both its restrictions to $(0,\infty)$ and $(-\infty,0)$ are smooth functions, given by
\begin{equation} \label{eq:defb1}
   { f(x) } = A^+ x^{\gamma}, \;\; x >0, \;\;\;\; \;\;\;\; \;\;\;\; \;\;\;\;  { f(x) } = - A^{-} (-x)^{\gamma},  x <0,
 \end{equation}
where $A^+$ and $A^-$ are fixed and strictly positive. We further assume that, as a distribution on $\RR$, $f$ is $\gamma$-homogeneous, namely
\begin{equation} \label{eq:defb2}
\forall \lambda >0,  f(\lambda \cdot) = \lambda^{\gamma} f(\cdot). 
\end{equation}

We will consider the solution $X^\eps$ to the SDE
\begin{equation} \label{eq:SDEmain}
X_t^{\varepsilon} =  \int_0^t f(X^\varepsilon_s) ds +  \eps W_t
\end{equation}
for a given continuous path $W$ and parameter $\eps > 0$. Since $f(X^\varepsilon)$ does not necessarily make sense pointwise, the precise definition requires some care (and is deferred to Subsection \ref{subsec:prelim} below), simply note that it implies in particular that $X^\varepsilon$ is continuous and for any sequence of smooth approximations $ {f^n \to_n f}$ $\in$ $C^\gamma$, it holds that,
\[ \forall t \geq s \geq 0, \; \; X^{\varepsilon}_t = X^{\varepsilon}_s + \lim_n\left( \int_s^t  {f^n}(X^\varepsilon_u) du \right) + \eps (W_t - W_s).\]
(In particular, $X^\varepsilon$ is a solution in the classical sense on any interval where it is non zero).

We will take $W$ as a sample path of a fBm with Hurst index $H \in (0,1)$, such that
\begin{equation}
\gamma > 1 - \frac{1}{2H},
\end{equation}
in which case it follows from the results of \cite{CG16} that, almost surely, a solution to \eqref{eq:SDEmain} exists and is unique.

We further let  $\varphi=\varphi(x,t) : (0,\infty) \times \RR_+ \to \RR$ be the semi-flow on $(0,+\infty)$ of the ODE $\dot{x} = \sgn{x} \absv{x}^\gamma$, namely $\varphi$ is such that
\[ \varphi(x,0) = x, \; \;  {\forall x>0,} \;\;\; \partial_t \varphi(x,t) = \varphi(x,t)^{\gamma} \]
and in fact one has the explicit formula 
\[ \varphi(x,t) = (1-\gamma)^{\inv{1-\gamma}} \left( x^{1-\gamma} + t \right)^{\frac{1}{1-\gamma}}  \]
which we extend by continuity to $x=0$.

Note then that the extremal solutions to $\dot{x} =  { f(x) }$ which exit $0$ immediately are simply given by
\[ \forall t \geq 0, \;\;\; x^{+,0}_t = \varphi(0,A^+ t), \;\;\; x^{-,0}_t = - \varphi(0,A^- t) .\]

Our main result is then as follows, where we recall from the introduction that
\[ t_{\eps} = \eps^{\frac{1-\gamma}{1-H(1-\gamma)}}, \;\;\;\;\;\;\;x_{\eps}=  \eps t_{\eps}^H = t_{\eps}^{\frac{1}{1-\gamma}}.\]

%
%

\begin{theorem}\label{thm:main}
Let $\alpha, \kappa$ be fixed to satisfy:

\[  { \kappa < \frac{2}{3} \left( \alpha \land 1 - H \right), \quad \alpha \in \left( H, \inv{1-\gamma} \right) } .\]

Then, for $X^\eps$ as above, there exists a random time $\psi_{\eps,\alpha,\kappa}$ such that:

\[ \forall t >0,\;\;\absv{ X_{ \psi_{\eps,\alpha,\kappa} + t} } \geq \left( \varphi\left( x_{\eps}, A^+ t \right) -  \eps t_{\eps}^{H-\alpha} t^{\alpha} \right) \indic{C^+_\eps} + \left( \varphi(x_{\eps}, A^- t ) -  \eps t_{\eps}^{H-\alpha} t^{\alpha} \right) \indic{C_{\eps}^-}, \]

where:

\begin{enumerate}
\item  $t_{\eps}^{\kappa} \ln \PP\left( \psi_{\eps,\alpha,\kappa} > z \right) \ls -z^{\kappa} $ ,
\item $C^+_{\eps}, C^-_{\eps}$ denote the two events such that:

\[ C^+_{\eps} = \{ X^{\eps}_{\psi_{\eps,\alpha,\kappa}} > 0 \}, \quad C^-_{\eps} = \{ X^{\eps}_{ \psi_{\eps,\alpha,\kappa} } < 0 \} .\]

In addition
\[ \PP(C^+_{\eps}) =  1 - \PP(C^-_{\eps}) \in (0,1) \mbox{ and does not depend on } \eps>0. \]
\end{enumerate}

\end{theorem}

 {
\begin{remark} One can see that the theorem above implies the simplified statement from the introduction (Theorem \ref{thm:main_simplified}), by taking $\nu = (H-\alpha) \frac{1-\gamma}{1+H(\gamma-1)} + 1 = \frac{1+\alpha(\gamma-1)}{1+H(\gamma-1)}$.
\end{remark}}

The theorem above gives a bound away from zero for the solutions. We also have the following upper bound (which is much simpler).

 {
\begin{proposition}\label{prop:upper_bound}
For $\gamma > 0$:
\[ \forall t \in [0,1]\,;-\varphi( 2 \eps \norm{ \wh }_{\infty;[0,1]}; A^- t ) \leq X^{\eps}_t \leq \varphi( 2 \eps \norm{ \wh }_{\infty;[0,1]}, A^+ t ).\]

For $\gamma  \leq 0$:
\[ \forall t \in [0,1],\;- \varphi\left( 0, A^- t \right) - 2 \eps \norm{ \wh }_{\infty;[0,1]}  \leq X^{\eps}_t \leq  \varphi\left( 0, A^+ t \right) + 2 \eps \norm{ \wh }_{\infty;[0,1]} . \]

\end{proposition}
}

\begin{remark}
The point (2) in Theorem \ref{thm:main} says that there is inherent randomness built in the system - namely that due to the possibility of the change of sign in $f$ the system does not converge to one particular path, but randomly picks either of the two extremal solutions.  {This phenomenon has been called ''spontaneous stochasticity'' in the physics literature, see for instance \cite{EB20, TBM20, DG98, Mai12} for a selection of articles on the topic. In a similar context to the problem investigated in the present article, see also the regularised models of blowup, for instance \cite{DM21, DMR20}.
}

In our context, this takes the following form : combining points (1), (2) one can see that the law of $X^{\eps}$ on $C([0,T])$ converges to the following atomic law:

 {
\[ p_+ \delta_{x^+} + p_- \delta_{x^-} \]
}
where $ {x^\pm}$ denote the two extremal solutions of  $\dot{x} =  { f(x) }$ starting from $0$, and the important observation is that both probabilities are strictly positive. In the Brownian case, the exact values of $p_+$ and $p_-$ are known (and depend on $A^\pm$ and $\gamma)$, but here we do not obtain any information other than both $p_+, p_-$ \blue{being} strictly positive  ({of course in the symmetric case $A_+=A_-$, $p_+=p_-=1/2$).}  {We expect that in general these values depend both on $H$ and $\gamma$ (see \cite{Lthesis} for some numerical experiments).} 

 {We also expect that in the fully asymmetric case where one of the coefficients is null, say $A_-=0$, then the limiting law concentrates on the extremal positive solution ($p_+=1)$, but this would require a more refined analysis than that of the present paper. Some partial results have been obtained in \cite{Lthesis} (in particular the convergence is proved when $\gamma > -1$).}
\end{remark}

\begin{remark}
We assume that $\gamma > 1 - \frac{1}{2H}$. In principle the heuristical scaling argument should apply up to the critical threshold $\gamma > 1 - \frac{1}{H}$, as long as one is able to show that, on small time scales, the drift is negligible in front of the noise (this is required in Stage 2 of the proof). However, in our proof, we use this fact in the rather strong form of Gaussian tail estimates (see Lemma \ref{thm:girsanov}), which at the moment seems to only be known in the literature for drifts $b \in C^\alpha$, $\alpha >  1 - \frac{1}{2H}$. 

Note that assuming  $\gamma > 1 - \frac{1}{2H}$ also allows us to use directly the result of \cite{CG16} on strong existence/uniqueness, but this is only a convenience, as the result would apply similarly if one only had weak existence. (However, in analogy with the $H=1/2$ case, one can in fact expect that strong well-posedness holds up to the critical threshold for power-type drifts in scalar equations).
\end{remark}

 {
\begin{remark}\label{rem:distribution_integer_gam}
We have written our assumption on $f$ in terms of its restriction on $\RR \setminus \{0\}$ \eqref{eq:defb1} and an homogeneity property \eqref{eq:defb2}, as these are the properties we use in the proof. We can in fact describe $f$ more explicitly (cf. \cite[Section 3.2]{Hor}). If $\gamma > -1$, then one simply has that $f$ is the (locally integrable) function
\[ f = \left( A_+ 1_{\{x>0\}} - A_-1_{\{x>0\}} \right) |x|^{\gamma}.
\]
If $\gamma < -1$ is not a negative integer, then \eqref{eq:defb1}- \eqref{eq:defb2} still define a unique distribution, namely
\[
f =  \frac{1}{(\gamma +1) \ldots (\gamma +k)} \left(\frac{d}{dx}\right)^{k} \Big(\left( A_+ 1_{\{x>0\}} - A_-1_{\{x<0\}} (-1)^k \right) |x|^{\gamma+k}\Big),
\]
where $\gamma + k >-1$ and the derivation is in the sense of distributions. If $\gamma=-k$ is a negative integer, things are less simple: a \blue{distribution} satisfying  \eqref{eq:defb1}- \eqref{eq:defb2} only exists if $A_+ = A_-$, and is then written in terms of Cauchy principal value and the Dirac mass at $0$ by
\[
f = \frac{A_+}{(-k)\ldots (-1)}  \left(\frac{d}{dx}\right)^{k-1} \left({\rm{p.v.}} \frac{1}{x} + C \delta_0 \right)\]
where $C \in \RR$ is arbitrary. Note that in that case $f$ is not characterized uniquely, which simply means that our results apply to all such $f$ regardless of the value of $C$ (although we may expect that the probabilities $p^{\pm}$ depend on $C$). 

Further note that in all cases, $f$ is smooth away from $\{0\}$, and satisfies a monotonicity property (increasing or decreasing depending on the sign of $\gamma$) on both $(-\infty,0)$ and $(0,\infty)$, which we will use crucially in our analysis. \end{remark}
}

\begin{remark}
Our result gives a sub-exponential rate of concentration around the extremal solutions : for instance, our result implies that, for any fixed $0< \delta < 1$, $T>0$, 
 { \[ \PP\left(  (1-\delta) x^-_T \leq X^\eps_T \leq (1-\delta) x^+_T + \right) \lesssim_{\eps \to 0} \exp\left( - C t_\eps^{-\kappa} \right),\] }
for any
\[ \kappa < \frac{2}{3} { \left(\frac{1}{1-\gamma} \land 1 - H \right) }. \]

We do not, however, expect this rate to be optimal. In the Markovian case ($H=1/2$ and $0 \leq \gamma \leq 1$), it is easy to see, using the scaling argument of \cite{DF14,PP18}, that one can take $\kappa =1$. This also follows from \cite{GHR01} (where they in addition obtain an actual logarithmic equivalent of this probability), since  $t_\eps^{-1}$ coincides with their (unusual) large deviation rate.

It would be interesting to obtain more precise concentration estimates in the fractional case. This is related to the (difficult) problem of speed of convergence in the long time behaviour of fractional SDE. Note that, in this context, it is known that the construction of Panloup-Richard \cite{PR20} (on which we build) leads to suboptimal (sub-exponential) rates. Indeed, under similar assumptions as them, Li-Sieber \cite{LS22} prove exponential rates of convergence (however, their methods do not seem applicable to the equation that we consider in this paper).
\end{remark}

Recall the transition point $(t_{\eps}, x_{\eps})$ :

\[ t_{\eps} = \eps^{\frac{1-\gamma}{1-H(1-\gamma)}}, \;\;\;\;\;\;\;x_{\eps}=  \eps t_{\eps}^H = t_{\eps}^{\frac{1}{1-\gamma}}.\]

As a consequence:
\begin{equation}\label{eq:transpoint}
\eps t_{\eps}^H x_{\eps}^{-1} = t_{\eps} x_{\eps}^{\gamma-1} = 1
\end{equation}

Before we start with a proof of the main theorem, we will give a brief proof of the scaling property. This will allow to vastly simplify the computations, because from now on we can work with $\eps = 1$, only later rescaling back to show our main claim - making the small noise limit \textsl{de facto} a corollary of our analysis of a long time behaviour of the system. 

This also means that in Section \ref{sec:main_section} we will show the construction of $\psi_{1,\alpha,\kappa}$, which we will denote simply $\psi_{\alpha,\kappa}$.

\begin{proposition} \label{prop:scaling}
For all $\eps > 0$ there holds:

\[ (X^{\eps}_{t t_{\eps}})_{t \geq 0} =_{\PP} (x_{\eps} X^1_t)_{t \geq 0} \]
\end{proposition}

\begin{proof}
By standard scaling properties of $\wh$:
 {
\begin{equation}\label{eq:rescaled}
d X^{\eps}_{t t_{\eps}} =_{\PP} t_{\eps} f(X^{\eps}_{t t_{\eps}} ) dt + \eps t_{\eps}^H d\wh_t 
\end{equation}
}

Using perfect scaling of $f$ and \eqref{eq:transpoint} we get:

\[ x_{\eps}^{-1} t_{\eps} f(X^{\eps}_{t t_{\eps}} ) = f(x_{\eps}^{-1} X^{\eps}_{t t_{\eps}} ) \]

The claim follows now by multiplying \eqref{eq:rescaled} with $x_{\eps}^{-1}$. 
\end{proof}

\subsection{Preliminaries}\label{subsec:prelim}

We consider equations of the form
\begin{equation} \label{eq:SDE}
X_t = x + \int_0^t  { f(X_r) } dr + w_t,
\end{equation}
where $f$ is possibly irregular. We first define the notion of solution that we use in the case where $f$ is not a continuous function (following \cite{CG16,GG20}).

For a given continuous path $w:[0,T] \to \RR$, and Schwartz distribution $f \in \mathcal{D}'(\RR)$, the averaged field $ { T^w f } \in C([0,T], \mathcal{D}'(\RR))$ is defined by 
\[  { T^w f }(t,\cdot) = \int_0^t f(\cdot + w_s) ds, \]

Assuming that $ { T^w f }$ is in $C_t^{\theta} C_x^{1}$ for some $\theta>\frac{1}{2}$, we say that $X:[0,T]\to \RR$ is a (path-by-path) solution to \eqref{eq:SDE} if $X_{\cdot}-w_{\cdot}$ is in $C^{\theta'}$ for some $\theta'>\frac 1 2$, and it holds that
\[ 
\forall t \geq 0, \;\; X_t =  x + \int_0^t (T^w f)(ds, X_s- w_s) + w_t ,
\]
where the integral is understood in the nonlinear Young sense.

We record here the well-posedness result that we will use. (Since our strategy of proof relies on restarting the equation in each step, we need to add an additive perturbation which is measurable w.r.t. the past, which does not create any difficulty).


\begin{lemma}\label{thm:girsanov}
Let $(B_t)_{t \geq 0}$ a Brownian motion adapted to a filtration $(\mathcal{F}_t)_{t \geq 0}$, and for $0<H<1$ let $N_t = \int_0^t (t-r)^{H-1/2} dB_r$, $t\geq 0$, and let $(P_t)_{t \geq 0}$ be a $\FF_0$-measurable path. Also fix $ { f \in C^{\gamma}(\RR)}$ with $\gamma > 1 - \frac{1}{2H}$. Then, a.s., $ {T^{N+P} f} \in C_t^{\theta} C_x^{1}$ for some $\theta> \frac{1}{2}$, and for any $x \in \RR$, the equation given by:

\[ X_t = x + \int_0^t  { f(X_r) } dr + N_t + P_t \]

has a unique path-by-path solution.  In addition, $(X_t)_{t \geq 0}$ is an $(\mathcal{F}_t)$-adapted process, and satisfies :

\[ \forall \lambda > 0\;\;  { \sup_{\norm{ g }_{\ccc{\gamma}} \leq 1}} \EE^{\FF_0}   \left[ \exp\left( \lambda  \norm{ \int_0^{\cdot} { g(X_r) } dr }_{\ccc{1+H\gamma}}^2 \right) \right] \leq C_\lambda, \]
where $C_{\lambda}$ does not depend on $P$.
\end{lemma}

\begin{proof}
It suffices to show that $t \mapsto N_t + P_t$ is locally non-deterministic:

\begin{equation}\label{eq:non_determinism}
\forall s<t\;\;\text{Var}^{\FF_s} \left( N_{st} + P_{st} \right) \geq  \EE^{\FF_s} N_{st}^2  = \EE^{\FF_s} \left( \int_s^t (t-r)^{H-1/2} dB_r \right)^2 \gtrsim (t-s)^{2H} 
\end{equation}

The proof then follows by the same arguments as in \cite{GHM22, GG20, CG16}.
\end{proof}

Before we set out to describe the dynamics in each of the three stages, let us gather in one place the parameters involved. The reader is invited to ignore this part during the first reading and come back to check the definition of a constant if need be. 
%
%
%
%
%



\begin{definition}\label{def:stage2_constants}
 {Let $A = \text{max}(A^+, A^-)$}. Define $t_e, c_{A,f} \in (0,1),  {t^* > 1}$, $U > 2, \vartheta \in (0,2), C_W \geq 1, \beta > H$ and $\delta >0$ to be constants satisfying the following relations:

\begin{equation}\label{eq:ca_tstar}
c_{A,f} \leq (1+t^*)^{H+\delta-1} \quad t^* = \left( U t_e^{-1/2-\delta} \right)^{\vartheta} \quad (t^*)^{10\delta/H} < 2
\end{equation}

\begin{equation}\label{eq:teu}
\max( \left(U t_e^{-1/2-\delta} \right)^{1+\vartheta(H-1/2-\beta)}, U^{-\vartheta(1/2-3\delta)} (t^*)^{-3\delta(1/2+\delta)},  c_{A,f} ) < K_A/3 
\end{equation}

\begin{equation}\label{eq:vartheta_lower}
M^*(t_e^{-1}, U) := t_e^{-(1/2+\delta)(1+\vartheta(H-1/2))} U^{\vartheta(H-1/2)} - t_e^H > 5^{\inv{\alpha(1-\gamma)}}
\end{equation}

\begin{equation}\label{eq:girsanov_proba}
C_B 12 (U-1) t_e^{-1/2} \exp\left( - t_e^{-1-2\delta} \right) - C_G \exp\left( - 9 t_e^{-1 - 2(1/2+H(\gamma-1))} \right) > 0 
\end{equation}

\begin{equation}\label{eq:escape_te}
(1 \lor A) t_e^{H \land 1/2+\delta} \leq 1/4 \land 8^{-\gamma} 
\end{equation}

\begin{equation}\label{eq:fixed_five}
 \begin{cases} q \in (2^{\inv{\alpha}}, 5^{\inv{\alpha}} ) & \gamma < 0 \\ q \in (1,1+\inv{2m_{\gamma}}) & \gamma > 0 \end{cases}
 \end{equation}
 
 where $m_{\gamma}$ is some positive constant depending only on $\gamma$ (see the proof of Lemma \ref{lem:curve_gam_pos})

\begin{equation}\label{eq:fixed_four}
 {K_A} = q^{-\alpha} \frac{5^{-\inv{\alpha(1-\gamma)}} \land A^{\alpha} \left( 1 + m_{\alpha,\gamma} \right)^{\inv{1-\gamma}}}{3} \quad m_{\alpha,\gamma} = \frac{\alpha(1-\gamma)}{1+\alpha(\gamma-1)} 
\end{equation}

\begin{equation}\label{eq:ell}
K_{\text{waiting}} C_W^{-\kappa+\delta} \leq  c_{A,f} 
 \end{equation}

where $C_G$ is the constant from Theorem \ref{thm:girsanov}  {with $\lambda = 1$}, $K_{\text{waiting}}$ is a constant from Lemma \ref{lem:failed_waitings_uniform} and $C_B$ is a constant that is chosen uniformly with respect to all the others in this Proposition (see the proof of Lemma \ref{lem:getting_out_proba}). 
We observe that we will always have $K_A <1$. 

\end{definition}

\begin{proposition}\label{prop:stage2_constants}
The combination of constants from Definition \ref{def:stage2_constants} exists.
\end{proposition}

\begin{proof}
See Section \ref{sec:technical}.
\end{proof}

\section{Construction of $\psi_{\alpha,\kappa}$}\label{sec:main_section}

Before we proceed with the construction, let us start with some preliminary notions. First we decompose the noise $\wh$ as follows, for any $s \leq t$ :

\begin{equation}\label{eq:new_past}
\wh_t - \wh_s = \int_s^t (t-r)^{H-1/2} dB_r + \int_{-\infty}^s ((t-r)^{H-1/2} - (s-r)^{H-1/2}) dB_r =: N_t + P_t .
\end{equation}

Define the kernel $G$ and compute its derivative with respect to time variable of the integrator.
\begin{equation}\label{eq:kernel}
\begin{aligned}
G(u,s,r) & = (s+u-r)^{H-1/2} - (s-r)^{H-1/2}, \\ \partial_r G(u,s,r) & = (1/2-H) \left( (s+u-r)^{H-3/2} - (s-r)^{H-3/2} \right).
\end{aligned}
\end{equation}

If we consider the increment $\wh_t - \wh_s$ for some $s<t$, then the past influence of the past sampled between $(u,v) \subset (-\infty, s)$ will be denoted as:
\begin{equation}
P_{t-s}^{(u,v),s} = \int_u^v G(t-s,s,r) dB_r.
\end{equation} We will often drop the second part in the upper index from the notation, assuming that $s=0$ - as shown in Lemma \ref{lem:failed_waitings_uniform} this dependence decays in time. 

We define the following condition to assure that the fluctuations of the past noise are small enough.  { In Section \ref{subsec:stage_three} we will show }that there exists a sequence of stopping times such that this is satisfied. 

\begin{definition}\label{def:admissibility}
 {Let $\alpha$ be as in Theorem \ref{thm:main}.} We will say that the admissible condition is satisfied with constants $c_{A,f}, c_{A,c}>0$ at time $\tau\geq 0$  if

\[ \sup_{t >0 } \sup_{s > 0} \frac{ P_s^{(-\infty,\tau-1),\tau+t} (1+t)^{ { (1-\alpha) \lor 0} }}{ s } \leq c_{A,f}, \quad \sup_{t > 0} \sup_{s>0} \frac{ P_s^{(\tau-1,\tau),\tau+t} }{ s^{H-\delta} } \leq c_{A,c} .\]

We will refer to the first bound above as ''admissibility of the remote past'', and to the second one as ''admissibility of the recent past''.

We will also denote that this condition is satisfied by writing  ''$P^{(-\infty,\tau)} \in \mathbf{A}( {c_{A,f},c_{A,c}})$''. 
\end{definition}

We also introduce the long term and short term past norms. One can show that they have Gaussian tails in all situations of our interest, that is $U<T$ are real-valued random variables, measurable with respect to $\FF_U$, see  {Proposition} \ref{prop:gaussian_tails_norms}.  {(Recall that the parameter $\delta > 0$ has been fixed in Definition \ref{def:stage2_constants})}.

\begin{definition}\label{def:long_short}
Let $0 < U < T \in \RR_+$ with $\absv{T - U } > 1$. We define long and short term norm as:

\[ L(U,T) = \sup_{r \in [U,T] } \frac{ \absv{B_r - B_T} }{\left(1 + T-r \right)^{1/2+\delta}} \quad S(T-1,T) = \sup_{r \in [T-1,T]} \frac{ \absv{B_r - B_T} }{ \absv{T - r }^{1/2-\delta} } .\]

\end{definition}

Now we can describe the construction of the random time $\psi_{\alpha,\kappa}$. We define it using a sequence of random variables $\Delta_j$ with known decay, defined precisely below \eqref{eq:kstar}. 

We set: 
\begin{equation}\label{eq:rho_definition}
\rho_{-1} = -\infty ,\quad \rho_0 = 0, \quad \forall k \geq 1,  \;\rho_k = \sum_{j=0}^{k-1} \Delta_j  
\end{equation}
and
\begin{equation} \label{eq:psi_definition}
\psi_{\alpha,\kappa} = \rho_{k^*} + 1+t^* 
\end{equation}
where $t^\ast$ is fixed above and $k^*$ is defined as:

\begin{equation}\label{eq:kstar}
k^* = \sup\{ k \in \NN: \Delta_j < \infty \}
\end{equation}

Before defining $\Delta_j$ let us define the following events: 

\begin{equation}\label{eq:stagetwo_events}
A_1^j = \{ P_{\cdot}^{(-\infty,\rho_j)} \in \mathbf{A}(c_{A,f}, 1-c_{A,f}) \} \quad A_2^j = \{ \absv{X_{\rho_j + 1 + t^*}} > 1, P^{(-\infty,\rho_j + 1 +t^*)}_{\cdot} \in \mathbf{A}(K_A, K_A) \}
\end{equation}

where  {$t^*$ is the fixed deterministic time that has been defined in Definition \ref{def:stage2_constants} and $c_{A,f}, K_A$ are some admissibility constants also defined there.}
\begin{definition}\label{def:deltas}
Let $C_W \geq 1$ be a parameter fixed in Definition \ref{def:stage2_constants}. We define $\Delta_0$ as: 
\begin{equation}\label{eq:delta_before_zero}
\Delta_0 = C_W+  L(-\infty,0)^{ {\frac{2}{\kappa}}}.
\end{equation}

For $j \geq 1$ we define $\Delta^1_j, \Delta^2_j, \Delta^3_j$ as follows: 

\begin{equation}\label{eq:delta1_def}
 \Delta^1_j = C_W j^{\inv{\kappa}-1}  + L(\rho_j-\Delta_{j-1},\rho_j)^{\frac{2}{\kappa}}
 \end{equation}
 
and 
 
\begin{equation}\label{eq:delta2_def}
\Delta^2_j = (1+t^*) \indic{ A_1^j }  ,
\end{equation} 
with $A_1^j$ as in \eqref{eq:stagetwo_events} and $t^*$ fixed above.

For any $j \geq 1$ we also define $\tau_{k,j}$ as $\tau_{k,j} = \rho_{j-1} + \Delta^1_j + \Delta^2_j + \sum_{m=1}^k q^m$ and then define $k_j^f$:

 {
\begin{equation}\label{eq:kjf_def}
k_j^f = \inf\{ k \in \NN: \norm{ B }_{\ccc{1/2-\delta};[\tau_{k,j},\tau_{k+1,j}]} > q^{k(\alpha-H)} \} 
\end{equation}
and 
}
\begin{equation} \label{eq:delta3_def}
\Delta^3_j = \indic{A_1^j \cap A_2^j} \sum_{m=1}^{k^f_j} q^m .
\end{equation}
%
%

 {
Finally we can define $\Delta_j, j \geq 1$ by
\[ \Delta_j = \Delta^1_j + \Delta^2_j + \Delta^3_j. \]
}

\end{definition}

Respectively the three $\Delta^i_j, i=1,2,3$ correspond to the following stages described in the introduction:

\begin{enumerate}
\item Time needed to put the system back into admissible condition
\item Time needed to get out of the critical strip
\item Time spent above the extremal curve.
\end{enumerate}

In the next three Subsections, we check that, at each step, the conditional probability of succeeding is bounded from below. We then obtain tail estimates on the $\Delta_j$ in Section \ref{sec:proof_main}, which allow us to conclude the proof of Theorem \ref{thm:main}.

 {
\begin{remark}
The choice of exponents in $L_j^{\frac{2}{\kappa}}, j^{\inv{\kappa}-1}$ becomes apparent later on in Proposition \ref{prop:tails_of_waiting_time}. In short, they are chosen in order to obtain a geometric bound on the waiting time, i.e. $\EE \exp\left( b \sum_{j=1}^k \Delta_j^{\kappa} \right) \leq (1+bC)^k$ for some constants $b,C > 0$. The parameter $\kappa$ is chosen as a result of interpolating different factors in waiting time, details of this computation are given in Lemma \ref{lem:failed_waitings_uniform}. The parameter $\alpha$ is then picked so that the past process does not grow faster than the flow of ODE, see Lemma \ref{lem:curve_gam_neg} and \ref{lem:curve_gam_pos}.
\end{remark}
}

\subsection{Stage 1}\label{subsec:stage_one}

The purpose of this Section is to prove the proposition below, which states that after Stage 1 the system is in an admissible state with positive probability (recall that $P$, $\mathbf{A}$ are defined in Definition \ref{def:admissibility}).

\begin{proposition}\label{prop:admissibility_probability}
Take $(\rho_j)_{j \in \NN}$ as in \eqref{eq:rho_definition}. Let $(\Delta_{j+1}^1)_{j \in \NN}$ be as defined in \eqref{eq:delta1_def}, $c_{A,f}$ as in Definition \ref{def:stage2_constants}. Then there exists $\lambda_1 > 0$ such that:

\[ \inf_{j \in \NN} \PP^{\FF_{\rho_j}}\left( P^{(-\infty,\rho_j+\Delta_{j+1}^1)} \in \mathbf{A}(c_{A,f}, 1-c_{A,f}) \right) \geq \lambda_1. \]
\end{proposition}

\begin{proof} 
Let us write $\tau_j=\rho_j+\Delta_{j+1}^1$ and decompose the past process as:

 {
\begin{equation}\label{eq:past_three_terms_first_prop}
P^{(-\infty,\tau_j),\tau_j} = P^{(-\infty,\rho_j),\tau_j} + P^{(\rho_j,\tau_j-1),\tau_j} + P^{(\tau_j-1,\tau_j),\tau_j} .
\end{equation}
}

By Definition \ref{def:admissibility} the first two terms will constitute the remote past and the last one will become the recent one. For the first one, by Lemma \ref{lem:failed_waitings_uniform} below, we have that, { almost surely,}
\begin{equation}\label{eq:cancellation}
\forall u,t >0\;\; u^{-1} (1 \lor t)^{- {(1-\alpha) \lor 0}} \int_{-\infty}^{\rho_j} G(u, \rho_j + \Delta^1_{j+1} + t, r) dB_r \leq K C_W^{-\kappa} \leq c_{A,f}/2,
\end{equation}
where the last inequality follows by our assumption on $C_W$.

We then consider the second and third terms. To this end, we want to show that, with positive probability, there holds:
\begin{equation}\label{eq:recent_past_admissibility_comp}
\begin{aligned}
\sup_{u,t>0}\;  (1+t)^{- {(1-\alpha) \lor 0} } u^{-1} \int_{\rho_j}^{\rho_j+\Delta^1_{j+1}-1} G(u, \rho_j + \Delta^1_{j+1}+t, r) dB_r & \leq c_{A,f}/2 ,\\ \sup_{u,t>0}  u^{-H+\delta} \int_{\rho_j+\Delta^1_{j+1}-1}^{\rho_j+\Delta^1_{j+1}} G(u, \rho_j + \Delta_{j+1}^1, r) dB_r \leq & 1/2.
\end{aligned}
\end{equation}
Note that if the above is satisfied, we sum \eqref{eq:cancellation} and the first term in \eqref{eq:recent_past_admissibility_comp} and obtain that the remote past satisfies Definition \ref{def:admissibility}, the recent past is immediately satisfied by the second inequality in \eqref{eq:recent_past_admissibility_comp}. 

We use Lemma \ref{lem:recent_past} below, to get that both terms in \eqref{eq:recent_past_admissibility_comp} are bounded by multiples of respectively  $L(\rho_j, \tau_j-1)$ and the second one with $S(\tau_j-1,\tau_j)$. Noting that these two quantities are independent conditionally on $\FF_{\rho_j}$ (by the strong Markov property), this yields, for some $c>0$, (which does not depend on $j$),

\begin{align*}
 \inf_j \PP^{\FF_{\rho_j}}\Big( P^{(-\infty,\rho_j+\Delta_{j+1}^1)} \in \mathbf{A}(c_{A,f}, {1-c_{A,f}}) \Big)&\geq \inf_j  \PP^{\FF_{\rho_j}}\Big( L(\rho_j, \tau_j-1) \leq c  \Big)  \PP^{\FF_{\rho_j}}\Big(S(\tau_j-1,\tau_j) \leq c \Big) \\
 & \geq \PP\left( \sup_{s \geq 0} \frac{|B_s|}{(s+1)^{1/2+\delta}}  \leq c\right) \PP\left( \norm{ B }_{\ccc{1/2-\delta};[0,1]} \leq c \right)  \\
 &=:  \lambda_1 > 0,
 \end{align*}
 where we have also used that $\tau_j-\rho_j := \Delta^1_j-1$ is $\FF_{\rho_j}$-measurable, and time reversal of Brownian motion.
\end{proof}

In the rest of this Section, we state and prove various results that were used in the preceding proof.

For any $a,b$ we can define the following quantity:
\[ M(a,b) = \inv{\absv{b-a}^{2\delta}} \sup_{u,v: a<u<v<b} \frac{ \absv{B_v - B_u}}{ \absv{v-u}^{1/2-\delta}}. \]
\begin{proposition}\label{prop:ibp}
Let $H \in (0,1)$. Then for any $s<t$, $u \geq 0$, there hold:

\begin{enumerate}
\item 
\[ \int_s^t (t+u-r)^{H-1/2} dB_r = \left( t + u- s \right)^{H-1/2} \left( B_t - B_s \right) + (H-1/2) \int_s^t (t+u-r)^{H-3/2} (B_r - B_t) dr. \]
\item 
\begin{equation}\label{eq:long_term_bound_wh}
\sup_{u \leq 1} \sup_{v \in [s,t]} \frac{ \absv{ \int_s^v (v+u-r)^{H-1/2} dB_r } }{ (1 + v-s)^{H+\delta}} \ls_H (1+t-s)^{2\delta} M(s,t).
\end{equation}

%
%
%
%
%
%
\item Recalling that $G(u,t,r):=(t+u-r)^{H-1/2} - (t-r)^{H-1/2}$, for all $s<t<\tau$ and $\gamma >\max(0, H-1/2) $ there holds:

\begin{equation}\label{eq:ibp}
\sup_{u >0}\; u^{-\gamma}  \int_s^t G(u,\tau,r) dB_r \lesssim \left( \absv{ \tau - s }^{H-\inv{2}-\gamma} \absv{ B_t - B_s } + \int_s^t \absv{\tau - r}^{H-\frac{3}{2}-\gamma} \absv{ B_r - B_t} dr \right).
\end{equation}
\vspace{0.25cm}
\item For $s = -\infty$ and $t=0$ and any $\sigma > H$:

\begin{equation}\label{eq:ibp_infty}
\sup_{u>0}\; u^{-\sigma}  \int_{-\infty}^0 G( u, \tau, r ) dB_r \lesssim \int_{-\infty}^0 \left( \tau - r \right)^{H-3/2-\sigma}  \absv{ B_r - B_0 } dr .
\end{equation}
\end{enumerate}

\end{proposition}

\begin{proof}
The first identity is just a standard integration by parts. For the second one we use the first one with $u \leq 1$. We bound the first term:

\[ \frac{ (v+u-s)^{H-1/2} \absv{B_v - B_s} }{(1+v-s)^{H+\delta}}  \leq  \frac{ (v+u-s)^{H-\delta}}{ (1+v-s)^{H+\delta}} (t-s)^{2\delta} \underbrace{ (t-s)^{-2\delta} \frac{ B_v - B_s}{ \absv{v-s}^{1/2-\delta}}}_{M(s,t)}. \]

For $u \leq 1$ the factor outside of $M(s,t)$ is then bounded with $\absv{s-t}^{2\delta}$. For the second term we obtain the same expression up to multiplicative constant $\inv{H-\delta}$.

For the third point it follows from the first one that
\begin{equation}\label{eq:bound_general}
\int_s^t G(u, \tau, r) dB_r \lesssim G( u, \tau, s) \absv{ B_t - B_s} + \int_s^t \absv{ \partial_r G(u, \tau, r) } \absv{ B_r - B_t } dr .
\end{equation}

We use the elementary fact that it holds for any $\theta \in (0,1)$, $\eta - \theta< 0$, $x > y>0$ :

\[ \absv{ x^{\eta} - y^{\eta} } \ls_{\eta,\theta} \absv{ x- y}^{\theta} \absv{ y}^{\eta - \theta}. \]
Therefore there holds 

 {
\begin{equation}\label{eq:g_holder}
G(\tau,u,s) \ls_{\gamma, H} u^{\gamma} \absv{ \tau - s}^{H-1/2-\gamma}, \;\;\; \partial_r G(\tau, u, r)  \ls_{\gamma, H} u^{\gamma} \absv{ \tau - r}^{H-3/2-\gamma}.
\end{equation}
}

This is sufficient to establish \eqref{eq:ibp}. 

The final identity follows by taking $\gamma = \sigma$, $t=0$ and letting $s\to -\infty$, noting that, since $\sigma>H$, the first term then converges to $0$ by standard growth properties of Brownian paths.
\end{proof}

We will use the following estimate on the recent past, recalling that $S$ and $L$ are defined in Definition \ref{def:long_short}.
\begin{lemma}\label{lem:recent_past}
For all $M>2$, $\tau, t_1 \geq 0$, it holds that 

\[ \sup_{u >0} u^{-1} P^{(\tau-M-t_1,\tau - 1 - t_1),\tau}_u \lesssim  (1+t_1)^{H-1+\delta} L(\tau-M - t_1,\tau-1-t_1), \]

\[ \sup_{u >0} u^{-H+2\delta} P^{(\tau-1-t_1,\tau-t_1), \tau}_u \lesssim S(\tau-1-t_1,\tau-t_1).  \]
\end{lemma}

\begin{proof}
The first inequality follows from \eqref{eq:ibp}, taking $\gamma= 1$, $s=\tau-M-t_1$, $t=\tau-1-t_1$. The second inequality also follows from  \eqref{eq:ibp}, taking $\gamma= H-2 \delta$, $s=\tau-1-t_1$, $t=\tau-t_1$.
\end{proof}

%
%
%
%
%
%
%
%
%
%
%
%

\begin{remark}
Note that the treatment of the most recent part is different from computations performed in \cite{PR20}. The needs are slightly different - in their case they could be satisfied with checking that the recent past is smaller than some positive quantity. In our case we need to take into account that the same fluctuations can be used as a driver for Riemann-Liouville fractional Brownian motion in Lemma \ref{lem:wait_above_before_the_climb} (stage 2) and the recent past in the stage 3 that follows (Lemma \ref{lem:infinite_brownian_holder_norms}), where the same noise is going to be part of $P^{(-T_{k+1}, -T_k)}_{\cdot}$ for all $k \in \NN$. This is why we need to obtain a bound directly in terms of the norm of underlying Brownian norm, which can be easily expressed.
\end{remark}

We use \eqref{eq:ibp} to estimate the rate of decay. 
\begin{lemma}\label{lem:failed_waitings_uniform}
Let $\Delta_j^1$ be as in \eqref{eq:delta1_def}. Then for a $C_W$ fixed in Definition \ref{prop:stage2_constants} there holds  {almost surely}:

\[ \sup_{k \in \NN} \sup_{u > 0} \sup_{t > 0} u^{-1}  \left( 1 + t \right)^{ {(1-\alpha) \lor 0}}  P^{(-\infty,\rho_k),\rho_k+\Delta_k^1+t}_u \leq K_{\text{waiting}}  C_W^{- {\kappa}} .\]

\end{lemma}


\begin{proof}
Let us define 

\begin{equation}\label{eq:theta_mu_params}
\xi_2 =  { \frac{2}{\kappa} } \inv{(1 - H - 2 \delta)}, \quad \xi_1 = 1 - \xi - \xi_2, \quad \xi = \frac{(1-\alpha ) \lor 0}{1-H-2\delta}.
\end{equation}
 {By the assumed condition on $\kappa$, we observe that

\[ \frac{ \left(1 - \alpha \right) \lor 0}{1-H-2\delta} + \frac{ 3\kappa}{2(1-H-2\delta)} < 1 .\]
}
Therefore holds that $\xi_1, \xi_2 > 0$ and  $\xi_1 > \frac{\kappa-\delta}{(1-H - 2\delta)}$.


We decompose as follows:

\[ \int_{-\infty}^{\rho_{k-1}} G(u,\rho_k + t, r) dB_r = \sum_{j=-1}^{k-2} \int_{\rho_j}^{\rho_{j+1}} G(u,\rho_k + t, r) dB_r  \] 

We first treat the term corresponding to $j=-1$, and using \eqref{eq:ibp_infty} we obtain
 {
\begin{align*}
\int_{-\infty}^0 G(u,\rho_k + t, r) dB_r \lesssim &  \;L(-\infty,0) \int_{-\infty}^0 \absv{\rho_k + t - r}^{H-3/2-1} (1-r)^{1/2+\delta} dr \\ \leq & \; L(-\infty,0) \absv{\rho_k+t-1}^{H-1+2\delta}  \int_{0}^{\infty} \absv{1+r}^{-3/2-2\delta} \absv{1+r}^{1/2+\delta} dr \\ \ls & \;L(-\infty,0) \absv{\rho_k+t-1}^{H-1+2\delta} \\
 \ls &\; L(-\infty,0)^{1+\frac{2}{\kappa}(H-1+2\delta)\xi_2} (1+t)^{(H-1)\xi} C_W^{(H-1+2\delta)\xi_1} \\
\leq &\; C_W^{(H-1+2\delta)\xi_1} (1+t)^{(H-1)\xi} ,
\end{align*}
}

where we used $\absv{\rho_k+t-r} > \absv{\rho_k+t-1} \lor \absv{1-r}$ in the second line, and in the fourth we used interpolation and the fact that
\[\rho_k + t - 1 \geq \Delta_{-1} + t - 1 \geq \max \left( L(-\infty,0)^{\frac{2}{\kappa}},   C_W, 1+t\right). \]

  To treat the summands corresponding to $j \geq 0$, we set $I_j := \int_{\rho_j}^{\rho_{j+1}} G(u,\rho_k + t, r) dB_r $ and we use \eqref{eq:ibp} to get that $I_j$ is bounded by :

\begin{equation}\label{eq:tau_to_rho}
\absv{ I_j } \ls_{H,1} \absv{ \rho_k + t- \rho_{j-1} }^{H-\inv{2}-1} \absv{ B_{\rho_j} - B_{\rho_{j-1}} } +  \int_{\rho_{j-1}}^{\rho_j} \absv{ \rho_k + t - r }^{H-\frac{3}{2}-1} \absv{ B_r - B_{\tau_{j}} }  dr .
\end{equation}

We will bound the two terms using the same estimate:

\begin{equation}\label{eq:lower_bd_tauk}
\forall r \in [\rho_{j-1},\rho_j],\;\;\rho_k + t - r \geq \left( \rho_k + t - \rho_j - 1 \right) + \left( 1 + \rho_j - r \right),
\end{equation}

which leads to:

\begin{equation}\label{eq:remote_past_bound}
 { \absv{I_j} } \ls_{H,1} L(\rho_{j-1},\rho_j) \absv{ \rho_k + t - \rho_{j-1} }^{H-1+2\delta} \left( 1 + \int_{\rho_{j-1}}^{\rho_j} \absv{\rho_k + t - r }^{-1-\delta} dr \right) .
\end{equation}

First we treat the bracket with integral inside - we have $\rho_k - \rho_j > 1$, which means that the integral is always bounded by some constant depending on $\delta$. To bound the term outside the bracket we then use \[ \rho_k - \rho_{j-1} + t \geq  t \lor L(\rho_{j-1}, \rho_j)^{ { \frac{2}{\kappa}} } \lor  C_W \sum_{m=j}^k m^{ { \inv{\kappa}-1}} .\] Here we ignore $t$ if $t < 1$. By interpolation, this leads to :

\[ \absv{ \rho_k - \rho_{j-1} }^{H-1+\delta} < (1 \lor t)^{(H-1+2\delta)\xi} C_W^{\xi_1(H-1+2\delta)} \left( \sum_{m=j}^{k-1} m^{\theta} \right)^{\xi_1(H-1+2\delta)}  L(\rho_{j-1}, \rho_j)^{ { \frac{2}{\kappa} } \xi_2(H-1+2\delta)}. \]

Putting it back into \eqref{eq:remote_past_bound} one observes again that the term in $L$ disappears, and to get the estimate in terms of $k$ we then compute

\[ \sum_{m=j}^{k-1} m^{ {\inv{\kappa}-1}} \gtrsim (k-1)^{ {\inv{\kappa}}} - (j-1)^{\inv{\kappa}} \gtrsim (k - j )^{ {\inv{\kappa}}}, \]

so that
 {
\[ \sum_{j=0}^k \left( \sum_{m=j}^k m^{ { \inv{\kappa}-1}} \right)^{\xi_1(H-1+2\delta)} < \sum_{j=0}^k (k-j)^{-( { \inv{\kappa}})(1-H-2\delta)\xi_1} \ls_{\theta,1} 1 .\]
}
 
where the uniform bound in the end follows by definition of $\xi_1$ and the assumption .

Summing over $k$, we obtain that

\[ \sum_{j=-1}^{k-2} \int_{\rho_j}^{\rho_{j+1}} G(u,\rho_k + t, r) dB_r \lesssim (1 + t)^{(H-1+2\delta)\xi} C_W^{\xi_1(H-1+2\delta)},  \]

which is the conclusion of the Lemma and by definition of $\xi_1$ we observe that $\xi_1(1-H-2\delta) < - \kappa$.
\end{proof}

\subsection{Stage 2}\label{subsec:stage_two}

This Section is devoted to the proof of the following Proposition.

\begin{proposition}\label{prop:stage2_prop}
Let $c_{A,f}, t^*, K_A$ be parameters defined in Definition \ref{def:stage2_constants}. Let $\tau$ be any stopping time such that the following holds a.s.\[ P^{(-\infty,\tau)} \in \mathbf{A}(c_{A,f},  {1-c_{A,f}}).\]

Then for the event $A_2$ defined by: 

\[  A_2= \left\{ X_{\tau+1+t^*}  > 1, \quad P^{(-\infty, {\tau+}1+t^*)} \in \mathbf{A}(K_A, K_A)\right\}, \]
 it holds that for some $\lambda_2 > 0$ \[ \PP^{\FF_{\tau}}\left( A_2 \right) \geq \lambda_2. \]

\end{proposition}

For brevity of notation we now take $\tau=0$ in the rest of this section. Since the proof's only probabilistic requirement is that $(B_t)_{t\geq 0}$ be a Brownian motion independent from $\FF_0$, the general case follows from the strong Markov property.

\begin{proof}
The dynamics can be decomposed in two steps. First in Lemmas \ref{lem:getting_out} and \ref{lem:getting_out_proba} we show that with positive probability if $X_0 \geq 0$, then $X_1 > 1$ (and $X_1 \leq -1 $ if $X_0 \leq 0$). Note that at this stage the admissibility is not necessarily satisfied, especially for $H<1/2$, see Remark \ref{rem:negative_correlations}. To fix this, we apply Lemma \ref{lem:wait_above_before_the_climb} - we wait an addional time $t^*$, after which both claims are satisfied with positive probability. The event $A_2 = \{ A_g, A_w \}$ with $A_g$ as defined in Definition \ref{def:getting_out_event} and $A_w$ as in Lemma \ref{lem:wait_above_before_the_climb}. Then:

\[ \PP^{\FF_{0}}\left( A_g, A_w \right) \geq \PP^{\FF_{0}} \PP^{\FF_{1}} \left( A_g, A_w \right) \geq \lambda_w \PP^{\FF_{0}}\left( A_g \right) \geq  \lambda_w \cdot \lambda_g =: \lambda_2. \] 

\end{proof}

We recall the following decomposition of Brownian motion $B$ on $[0,1]$ :

\begin{equation}\label{eq:bm_bridge}
\forall t \in [0,1]\;\; B_t = t Z + b_{t},
\end{equation}

where $Z \sim \mathcal{N}(0,1)$ and $b$ is a Brownian bridge, and $b$, $Z$ are independent. To construct a suitable path $w_t$ we will show that with positive probability two events take place - the first one, $A_g$, described in Lemma \ref{lem:getting_out_proba} ($A_g$ coming from \textsl{getting out}) is an event such that  {the new noise allows to apply Lemma \ref{lem:getting_out},} under which $|X_1|$ is sufficiently far from $0$. The second one, $A_w$ (coming from \textsl{waiting}) is a technical step, needed to cool down the negative influence of this noise once it becomes the part of the past, see estimates in Lemma \ref{lem:bridge_in_the_past}.

We will start from the preparatory lemma about the lower bound on $\wh$, which is going to allow us to compute the probability of satisfying deterministic estimates.

\begin{lemma}\label{lem:new_noise_bridge_decomp}
Let $Z, b$ be as in \eqref{eq:bm_bridge}:

\[ \forall s \in [0,1]\;\;\int_0^s (s-r)^{H-1/2} dB_r = Z {s}^{H+1/2} C_{H} + b^H_s ,\]

where $\norm{ b^H }_{\ccc{H-\delta} {;[0,1]}} \leq c_{H,b} \norm{ b }_{\ccc{1/2-\delta} {;[0,1]}}$. Moreover, for the same $(Z,b)$ there holds

\[ \int_0^1 (1+z-r)^{H-1/2} dB_r =  c_H h(z) Z + b^{H,z}_{1,0} ,\quad h(z) > \begin{cases} (1+z)^{H-1/2} & H < 1/2 \\ z^{H-1/2} \lor 1 & H > 1/2 \end{cases}  , \]

with $\sup_{z > 0} (1+z)^{3/2-H} \absv{ b^{H,z}_{1,0} } \leq C_H \norm{ b }_{\ccc{1/2-\delta};[0,1]}$ and $C_H = \frac{ \absv{H-1/2}}{H-\delta}$.
\end{lemma}

\begin{proof}
See Section \ref{sec:technical}.
\end{proof}

For the same $(Z,b)$ as in Lemma \ref{lem:new_noise_bridge_decomp} we have the following:

\begin{lemma}\label{lem:bridge_in_the_past}
For $w > 1, H \in (0,1), \beta \in (H,1]$:

\begin{equation}\label{eq:bridge_past_lower_bd}
\sup_{u > 0} u^{-\beta} \absv{ \int_0^1 (1+w+u-r)^{H-1/2} - (1+w-r)^{H-1/2} dB_r }  <  w^{H-1/2-\beta} \left( 2Z + \norm{ b }_{\ccc{1/2-\delta};[0,1]} \right).
\end{equation}

\end{lemma}

\begin{proof}
See Section \ref{sec:technical}.
\end{proof}

Then we can define the desired event.

\begin{definition}\label{def:getting_out_event}
Let $ (Z_g, b_g)$ be the decomposition of $(B_t)_{t \in [0,1]}$ as in \eqref{eq:bm_bridge}. Let $A_g$ be defined as the set on which it holds that
\begin{enumerate}
\item  $ c_{H,b} \norm{ b_g }_{\ccc{1/2-\delta}} \leq 1$, where $c_{H,b}$ is the constant from Lemma \ref{lem:new_noise_bridge_decomp}.

\item$Z_g \in (12 t_e^{-1/2-\delta}, 12 U t_e^{-1/2-\delta})$, where $t_e,U$ are fixed in Definition \ref{def:stage2_constants}.

\item If $\gamma < 0$, the  {random variable $V$} $= \sup_{0 \leq u \leq 1} u^{-1-H\gamma+\delta} \absv{ \int_0^u  { f(X_r) } dr }$ satisfies \eqref{eq:drift_constant_bound} { below}.
\end{enumerate}
\end{definition}

\begin{remark}
Let us clarify the definition of the integral in the third point in Definition \ref{def:getting_out_event}. Initially, when $X$ cannot be bounded away from zero, one can understand this integral in the same sense as \cite{CG16}, namely as a nonlinear Young integral of the averaged field :

\[ \int_0^t  { f(X_r) } dr :=   \int_0^t (T^{W} f)(dr, X_r- W_r) . \]

 We bound this quantity with probabilistic arguments using Theorem \ref{thm:girsanov} for all $t < t_e$, where $t_e$ is as fixed in Definition \ref{def:stage2_constants}. For $t > t_e$ we show that $X$ is bounded away from zero, which implies that this integral can be understood classically because $ { f(x) }$ is smooth away from $0$.
\end{remark}

First we will show that under a deterministic conditions on the path $w$ our trajectory $X$ leaves the critical strip. The constant $t_e$ is a small deterministic time, fixed in Definition \ref{def:stage2_constants}. 

\begin{lemma}\label{lem:getting_out}
Let $X_0 \in (-1,1)$. Let $(w_t)_{t \in [0,1]}$ satisfy:

\begin{equation}\label{eq:noise_decomposition_deterministic}
w_t = Z t^{H+1/2} + u_t  \quad \norm{ u }_{\ccc{H-\delta}} < 2
\end{equation}

If $X_0 \geq 0$ (respectively $< 0$) then assume that, for $t_e$ as defined in Definition \ref{def:stage2_constants}

\begin{equation}\label{eq:normal_and_bridge_assumption}
Z \geq 12 t_e^{-1/2-\delta}
\end{equation}

(respectively $Z \leq - 12 t_e^{-1/2-\delta}$). 

For $\gamma < 0$ we additionally assume that:

\begin{equation}\label{eq:drift_constant_bound}
V :=   \sup_{0 \leq u \leq 1} u^{-1-H\gamma+\delta} \absv{ \int_0^u  { f(X_r) } dr }\leq 9 t_e^{-(1+H(\gamma-1)-\delta)}
\end{equation}

Then there holds:

\[  X_1 \geq t_e^{-1/2-\delta} \quad \text{(respectively $X_1 \leq -t_e^{-1/2-\delta}$)} \]

\end{lemma}

\begin{proof}
  { For convenience we will use the upper bound for $A^{\pm}$ in computations for both cases, to make generalization as obvious as possible, which we will denote by $A$, i.e. $A = \max( A^+, A^- )$.} With no loss in generality we assume that $X_0 \geq 0$, and in fact by monotonicity of the flow we can simply assume $X_0=0$.

First we treat the case of $\gamma > 0$. Set $\norm{ X }_{\infty;t} := \sup_{r <t } \absv{X_r}$. We claim that for $t < t_e$:

\begin{equation}\label{eq:x_sup_est}
\norm{ X }_{\infty;t} \leq 8 t^{H-\delta}  + 2| Z| t^{H+1/2}
\end{equation}

We use Young inequality $ab \leq \gamma a^{1/\gamma} + (1-\gamma) b^{1/(1-\gamma)}$ for $a,b>0$ between first and the second line: 
\begin{align*}
\norm{ X}_{\infty;t} & \leq A t \norm{ X }_{\infty;t}^{\gamma}  + |Z |t^{H+1/2} + 2 t^{H-\delta} \\
& \leq A t^{1+H(\gamma-1)} \left( (1-\gamma) t^{H} + \gamma \norm{ X}_{\infty;t} \right) + Z t^{H+1/2} + 2 t^{H-\delta} \\
& \leq A(1-\gamma) t^{1+H\gamma} + 2 t^{H-\delta} + \norm{ X }_{\infty;t} t^{1+H(\gamma-1)} A  + Z t^{H+1/2}
\end{align*}

Using \eqref{eq:escape_te}, we have for $t \leq t_e$, \[ A(1-\gamma) t^{1+H\gamma} + 2 t^{H-\delta} = 2t^{H-\delta} \left( 1 + A(1-\gamma) t^{1+H(\gamma-1)+\delta} \right) \leq 4 t^{H-\delta} \]

Upon reusing the same inequality  for the term with supremum norm in the chain of inequalities above, the inequality \eqref{eq:x_sup_est} follows at once. It then follows that within the range of applicability of \eqref{eq:x_sup_est} we have:

\[ \absv{ \int_0^t \sgn{X_r} \absv{ X_r }^{\gamma} dr } \leq 8^{\gamma}t^{1 + H \gamma - \gamma \delta} + 2^{\gamma} |Z|^{\gamma} t^{1 + \gamma(H+1/2)}. \]

We consider again $t < t_e$ and obtain\begin{equation}\label{eq:x_te_bound}
\begin{aligned}
\forall t < t_e,\;\;X_t \geq &  - A 8^{\gamma} t^{1 + H \gamma - \gamma \delta} - A 2^{\gamma} Z^{\gamma} t^{1 + \gamma(H+1/2)} + Z t^{H+1/2} - 2 t^{H-\delta} \\
\geq & - A 8^{\gamma} t^{1 + H \gamma - \delta} + Z t^{H+1/2} \left( 1 - 2^{\gamma} A Z^{\gamma-1} t^{1+ H(\gamma-1)} \right) - 2 t^{H-\delta} \\
\geq & Z t^{H+1/2} /2 - A 8^{\gamma} t^{1+H \gamma - \delta} - 2 t^{H-\delta} > Z t^{H+1/2} /2 - 4 t^{H-\delta} 
\end{aligned}
\end{equation}
where between the second and third line we have used \eqref{eq:escape_te} again the term in the bracket is bounded from below by $1/2$ - (one bounds $Z^{\gamma-1} < 1$ and $A t^{1+H(\gamma-1)} < 1/2$ for all $t < t_e$), and we also use this assumption in the last inequality.

We then use \eqref{eq:normal_and_bridge_assumption} and obtain
\[ X_{t_e} \geq 12 t_e^{H-\delta} /2 - 4 t^{H-\delta} \geq 2t_e^{H-\delta}. \] 
The goal is now to show that this also holds for every $t > t_e$, which will allow us to ignore the drift - if $X$ will not change sign, then 
we will have $\inf_{t \in [t_e,1]} \int_{t_e}^t  { f(X_r) } dr > 0$. To this end, we will show that:

\begin{equation}\label{eq:noise_not_too_bad}
\forall t > t_e, \;\; w_t - w_{t_e} \geq -  2t_e^{H-\delta} 
\end{equation}

Note that for $t < 2 t_e$ the inequality is immediately true, because we estimate the path $u_t$ using its H\"older bound, so that $\forall t \in (t_e, 2t_e)\;\;\absv{ u_t - u_{t_e} } < 2 t_e^{H-\delta}$, therefore we need to show that it holds for the remainder of the interval, that is $[2t_e, 1]$. We can then write the lower bound:

\[ w_t - w_{t_e} \geq Z \left( t^{H+1/2} - t_e^{H+1/2} \right) - \absv{ t - t_e }^{H-\delta} 2 =: h(t) \]

As announced, we set out to show $h'(t) > 0$ for $ t> 2t_e$. A simple computation yields:

\[ h'(t) = Z (H+1/2) t^{H-1/2} - (t-t_e)^{H-1-\delta} 2 \]

 {
Let $s \geq 1$ and set $t = s t_e$. Then we obtain the lower bound of \[ h'(st_e) > 12 (H+1/2) t_e^{H-1-\delta} s^{H-1/2} - 2(s-1)^{H-1-\delta} t_e^{H-1-\delta}  \] After rearranging it amounts to checking that for $s \geq 2$ there holds \[ s^{H-1/2} (s-1)^{1-H+\delta} > s^{1/2+\delta} 2^{H-1-\delta} > \inv{6(H+1/2)} \]  One checks that the rightmost inequality is satisfied for $s^{1/2+\delta} \geq 2/3$, \blue{and in particular under the assumed condition that $s \geq 2$}. Therefore $h'(t) > 0$ for any $t > 2t_e$ and $\inf_{t \in [t_e, 1]} h(t) > - 2t_e^{H-\delta}$. These two together yield \eqref{eq:noise_not_too_bad}. Therefore we can write for any $t > t_e$, using $\int_{t_e}^t  { f(X_r) } dr > 0$:
}

\begin{equation}\label{eq:xt_bound}
X_t  = X_{t_e} + X_t - X_{t_e} \geq t_e^{H-\delta} + Z \left( t^{H+1/2} - t_e^{H+1/2} \right) - 2 \absv{ t - t_e }^{H-\delta} 
\end{equation}

We have just established that the right hand-side is positive. Moreover for $t=1$ one has \[ X_1 \geq Z \left( 1 - t_e^{H+1/2} \right) - 2 \geq 10 t_e^{-1/2-\delta} \cdot \left( 1 - t_e^{H} - 2 t_e^{1/2+\delta} \right) \geq t_e^{-1/2-\delta}\]

which holds due to the assumption on $t_e$.

For the case of $\gamma < 0$ instead of buckling the inequality to get the estimate on $\norm{ X }_{\infty}$, we use the third point from Definition \ref{def:getting_out_event} and the assumption \eqref{eq:drift_constant_bound}. This yields:

\[ X_t \geq - V t^{1+H \gamma -  \delta} + Z t^{H+1/2} + u_t \]

Applying the same computations to noise dependent terms we have:

\[ X_{t_e} \geq - V t_e^{1+H\gamma - \delta} + Z t_e^{H+1/2} - 2 t_e^{H-\delta} > 10 t_e^{H-\delta} - V t_e^{1+H \gamma - \delta} \]

We impose then $V < 7 t_e^{-\left(1 + H(\gamma-1)- \delta \right)}$. We now follow the identical reasoning to show that $X > 0$ for all $t > t_e$ and the proof of the first claim is finished.
\end{proof}

\begin{lemma}\label{lem:getting_out_proba}
For the event $A_g$ as defined in Definition \ref{def:getting_out_event} there holds: 

\begin{enumerate}
\item If the past process $P_t := P^{(-\infty,0)}_t$ is in the admissible state  {with $c_{A,f} + c_{A,c} \leq 1$}, then the conditions of Lemma \ref{lem:getting_out} are met for  $\wh_t = Z_g t^{H+1/2} + b^H_t + P_t$.
\item $A_g$ is independent from $\FF_0$,
\item there exists $\lambda_g > 0$ such that $\PP^{\FF_0}\left( A_g \right) \geq \lambda_g$.
\end{enumerate}

\end{lemma}

\begin{proof}
The second claim is immediate by the independence of Brownian increments. We can decompose $\wh$ as follows - the last two terms constitute the assumed path $(u_t)_{t \in [0,1]}$.

\[ \wh = Z t^{H+1/2} + b^H_t + P_t \]

By the admissible condition we have  { \[ \norm{ P }_{\ccc{H-\delta};[0,1]} \leq \norm{ P^{(-\infty,-1)} }_{\ccc{1};[0,1]} + \norm{ P^{(-1,0)} }_{\ccc{H-\delta};[0,1]} \leq 1 \] }and by Lemma \ref{lem:new_noise_bridge_decomp} we have $\norm{ b^H }_{\ccc{H-\delta}} \leq c_{H,b} \norm{ b }_{\ccc{1/2-\delta}} \leq 1$, which yields $\norm{ u }_{\ccc{\alpha}} \leq 2$ as assumed in Lemma \ref{lem:getting_out}.

For $\gamma > 0$, 
it suffices to get the bounds on $Z$ and $\norm{ b }_{\ccc{1/2-\delta};[0,1]}$ and we get

\[ \PP^{\FF_0}\left( A_g \right) \geq \PP\left( Z > C t_e^{-1/2-\delta}, \norm{ b }_{\ccc{1/2-\delta}} < c_{H,b}^{-1} \right) = \lambda_g \]

and the fact that $\lambda_g > 0$ is easy to see using independence of $Z$ and $b$.

For $\gamma < 0$ we need to include a bound on  {$V$}. 
 By the  {Gaussian tails of $V$} (Theorem \ref{thm:girsanov})  {with \blue{$\lambda = \norm{ f }_{\ccc{\gamma}}$}} we  {observe that there exists a constant $C_G$ such that}: 

\begin{equation}\label{eq:gaussian_tails_drift}
\EE^{\FF_0} \exp\left(  V^2 \right) \leq C_G 
\end{equation}

 {where the lower index in $C_G$ comes from \textsl{gaussian}}. With this bound at hand we can estimate the probability of the event $A_g$ using Gaussian tails of $V$ - the condition in Lemma \ref{lem:getting_out} reads \[ V < 9 t_e^{-(1+H(\gamma-1)-\delta)} \]

We then compute the probability of the event $A_g$ - $C_B$ below is the probability of the bridge part, which is independent from $t_e$.

\begin{equation}\label{eq:gamma_neg_escape_event}
\begin{aligned}
\PP\left( A_g \right) & = \PP\left( Z \in (12t_e^{-\delta-1/2}, 12 U t_e^{-1/2-\delta}), \norm{ b } < c_H^{-1}, V < 9 t_e^{-(1+H(\gamma-1) - \delta)} \right) \\ & > C_B 12 (U-1) t_e^{-1/2} \exp\left( - t_e^{-1-2\delta} \right) -  { C_G \exp\left( - 9 t_e^{-1 - 2(1/2+H(\gamma-1))} \right) },
\end{aligned}
\end{equation}
and the result follows since we have assumed that this quantity was positive, see \eqref{eq:girsanov_proba}.
\end{proof}


\begin{remark}\label{rem:negative_correlations}
An important drawback of this approach is that for $H < 1/2$ we cannot show that the noise is in the admissible state right after escaping the strip $[-1,1]$. This comes from considering the restarted noise, that is:

\[ t>1\;P^{(0,1),1}_t = \int_0^1 (t-r)^{H-1/2} - (1-r)^{H-1/2} dB_r = \int_0^1\left( (t-r)^{H-1/2} - (1-r)^{H-1/2} \right) Z dr + b^H_t. \]

The first part is clearly going to be negative and one gets the bound 
\begin{equation}\label{eq:z_contrib_approx}
t^{H-\delta} 2 \absv{Z} \sim 4 t^{H-\delta}  t_e^{-1/2}.
\end{equation} Since we have an interest of taking $t_e$ very small and the admissibility constant is smaller than one, so in particular smaller than the right hand-side of \eqref{eq:z_contrib_approx}, then it follows that the H\"older constant of the first term is too big. In other words, when the positive noise in $(0,1)$ becomes part of the past, it influences the trajectory in the opposite way. This is coherent with negative correlations of increments of $\wh$ in this case. Indeed, consider for simplicity a Riemann-Liouville fBm, that is $\wi_t = \int_0^t (t-r)^{H-1/2} dB_r$ and write for some $t > s$

\[ \EE \left[ \wi_s \left( \wi_t - \wi_s \right) \right] = \int_0^s \left( s- r \right)^{H-1/2} \cdot \left( \left( t- r \right)^{H-1/2} - \left( s - r \right)^{H-1/2} \right) dr \]
 
 which for $H<1/2$ is clearly negative. At the same time, if we pick some $u>1$ and instead compute $\EE \left[ \wi_s \left( \wi_{t+u} - \wi_{s+u} \right) \right]$, then this dependence will decay up to absolute value as $\absv{t-s}^{\beta} u^{H-\frac{1}{2}-\beta}, \beta > 0$.
\end{remark}

\begin{lemma}\label{lem:wait_above_before_the_climb}
Let $t^*, c_{A,f}, t_e, U, K_A$ and $\vartheta$ be as fixed in  {Definition \ref{def:stage2_constants}}. Let $M^*(t_e^{-1}, U)$ be the function defined in \eqref{eq:vartheta_lower} and recall that $M^*(t_e^{-1},U) > 1$. If we define 
\begin{equation}\label{eq:waiting_second_part_event}
 { A_s \quad = \quad \left\{ X_{1+t^*} > M^*(t_e^{-1},U), \;\;\; P^{(-\infty,1+t^*)} \in \mathbf{A}(K_{A^+}, K_{A^+}) \right\}},
\end{equation}

then there exists $\lambda_w$  {$>0$ } such that 
\begin{equation*}
\PP^{\FF_1}\left(  { A_s} \right) \geq \lambda_w \mbox{ a.s.} \mbox{ on }A_g.
\end{equation*}
\end{lemma}

\begin{proof}

After Lemma \ref{lem:getting_out} we may have lost admissibility after leaving the critical strip (Remark \ref{rem:negative_correlations}). We claim that with this choice of $t^*$ one can with positive probability sample Brownian increments such that admissibility is going to be recovered. We recall that $K_A := K(A,1)$ is taken to be uniform in other constants appearing in Definition \ref{def:stage2_constants}.

The constant $t^*$ is chosen such that almost surely there will hold:

\begin{equation}\label{eq:decay_new_past}
\forall u,s >0\;\;u^{-1} (1+s)^{H+2\delta-1} \absv{ P^{(0,1),1+t^*+s}_u }  \leq K_A/3,
\end{equation}

and will construct the event $A_w$ such that

\begin{equation}\label{eq:decay_new_noise}
\forall u,s > 0\;\;u^{-1} \left( 1 + s \right)^{H-1} \absv{ P^{(1,t^*),1+t^*+s}_u } \leq K_A/3 \quad u^{-H+\delta} P^{(t^*,1+t^*),1+t^*+s}_u \leq 1.
\end{equation}

In what follows we will show that for  { $A_s$ as defined in \eqref{eq:waiting_second_part_event} we have $A_s \supset A_w$ } and further that there exists $\lambda_w$ such that $\PP^{\FF_1}\left( A_w \right) > \lambda_w$ a.s. on $A_g$.  { Precise definition of $A_w$ is given later in \eqref{eq:aw_def}.}

Let us outline how the second one part of $A_s$ holds if we sample the noise satisfying \eqref{eq:decay_new_past} and \eqref{eq:decay_new_noise}. Automatically the recent part of admissibility is satisfied, so that we need to check the remote part. Assuming that \eqref{eq:decay_new_past} and \eqref{eq:decay_new_noise} hold and using the admissibility of $(-\infty,0)$ - (recall that at the beginning of Stage 2 in Proposition \ref{prop:stage2_prop} we assumed that $P \in \mathbf{A}(c_{A,f}, c_{A,c})$ with $c_{A,f} < 1/3$) one obtains:

 {
\begin{align*}
 P^{(-\infty,t^*),1+t^*+s}_u \leq & u^{1} (1+s)^{(1-\alpha) \lor 0} c_{A,f} + u^{1} (1+s)^{H+2\delta - 1} K_A /3 + u^{1} (1+s)^{H-1} K_A / 3
 \\ \leq & u^{1} (1+s)^{(1-\alpha) \lor 0}  K_A,
 \end{align*}
}

where the second inequality  {follows by observing that the first term is going to dominate the others and gathering the constants.} 

\paragraph{\textbf{Choice of $t^*$ and the proof of \eqref{eq:decay_new_past}}}
Initially consider the case $s=0$ in \eqref{eq:decay_new_past}, which follows by the second claim in Lemma \ref{lem:bridge_in_the_past} taken with $\beta = 1$, setting $\absv{Z } < U t_e^{-1/2-\delta}, U > 1$ and $w = 1+t^* = U^{\vartheta} t_e^{-(1/2+\delta)\vartheta}$. The parameter $\vartheta$ is fixed to satisfy: 

\begin{equation}\label{eq:decay_negative}
U t_e^{-1/2-\delta} U^{\vartheta(H-1/2-1)} t_e^{-\vartheta(H-1/2-1)(1/2+\delta)} \leq K_A/3.
\end{equation}

This is the same constraint as the first one in \eqref{eq:teu} but with $1$ instead of $\beta$. It is shown in the proof of Definition \ref{def:stage2_constants} that $\beta > \inv{\vartheta} + H - \inv{2}$ and we assumed there that $1 \geq \beta$. (see Remark \ref{rem:vartheta_inftwo} below). To obtain the decay for $s > 1$ we take $w = U^{\vartheta} t_e^{-(1/2+\delta)\vartheta} + s$, $\beta = H+2\delta$ in Lemma \ref{lem:bridge_in_the_past} and then replace \eqref{eq:decay_negative} with interpolation:

\[ U t_e^{-(1/2+\delta)} w ^{H-1/2-1} \leq U t_e^{-(1/2+\delta)} \left( U^{\vartheta} t_e^{-(1/2+\delta)\vartheta} \right)^{1/2-2\delta} (1 \lor s)^{H+2\delta -1} \leq K_A / 3 \cdot (1 \lor s )^{H+2\delta-1}, \]

where we have used \eqref{eq:teu} in Definition \ref{def:stage2_constants}.

\paragraph{\textbf{The proof of $\inf_{t \in [1,1+t^*]} X_t > 1$ }}
\smallskip
We will now detail the estimates necessary to show that on the event $A_w$ the path $X_t$ remains bounded away from one and as a byproduct we will prove that \eqref{eq:decay_new_noise} holds with positive probability. To this end, let us first decompose the dynamics of $X_t$ for $t \in [1,1+t^*]$. We have, making use of the fact that the for all $r \in [t_e,1]$ one has $X_r > 0$ (cf. the proof of Lemma \ref{lem:getting_out}):

\begin{equation}\label{eq:xt_lower_bd}
X_t \geq \int_0^1 (t-r)^{H-1/2} dB_r +  \int_0^{t_e}  { f(X_r) } dr + \int_1^t (t-r)^{H-1/2} dB_r + \int_1^t  { f(X_r) } dr + P^{(-\infty,0),0}_t .
\end{equation}

In what follows we will use this decomposition to show that $\inf_{t \in [1,1+t^*]} X_t > 1$. For $t \geq 1$ one exploits Lemma \ref{lem:new_noise_bridge_decomp} with $z = t-1$ therein and obtains:

\begin{equation}\label{eq:new_noise_past_one}
\begin{aligned}
\int_0^t (t-r)^{H-1/2} dB_r = & \int_0^1 (t-r)^{H-1/2} dB_r + \int_1^t (t-r)^{H-1/2} dB_r \\
= & Z h(t) +  { b^{H,t-1}_{1,0} } + \int_1^t (t-r)^{H-1/2} dB_r 
\end{aligned}
\end{equation}

 where $h(t) \simeq t^{H-1/2} \lor 1$ is as given in Lemma \ref{lem:new_noise_bridge_decomp}.  {Similarly, as denoted there $b^{H,z}_{1,0} = \int_0^1 (1+z-r) db_r$}. We aim at estimating the worst case scenario for the noise, namely:

\begin{equation}\label{eq:worst_case}
 \inf_{t \in [1,1+t^*]} \wh_t = \inf_{t \in [1,1+t^*]} \left[ Z h(t) +  { b^{H,1}_{1,0} } +  \int_1^t (t-r)^{H-1/2} dB_r + P_t^{(-\infty,0)} \right]  > 1.
 \end{equation}  
 
By the assumptions in Definition \ref{def:getting_out_event} and from the last claim of Lemma \ref{lem:new_noise_bridge_decomp}, where the bound is uniform in $t$ one has:
 \begin{equation}\label{eq:bridge_bound}
 \sup_{t \in [1, 1+t^*]}  { \absv{ b^{H,t-1}_{1,0} } } < c_H \norm{ b_g }_{\ccc{1/2-\delta}} \leq 1.
 \end{equation} For any $a<b$ we can define for $\delta > 0$ fixed before:

\begin{equation}\label{eq:mnorm}
M(a,b) = \inv{\absv{b-a}^{\delta}} \sup_{a<u<v<b} \frac{\absv{B_v - B_u}}{ \absv{v-u}^{1/2-\delta}} .
\end{equation}

The scaling is chosen such that $M(a,b)$ has the same distribution for any values of $a,b$. 
We will now show that:
 
\begin{equation}\label{eq:rl_bound}
 { \forall t \in [1, 1+t^*], }\; \absv{  \int_1^t (t-r)^{H-1/2} dB_r } \ls_H (1+t)^{H+\delta} M(1,t^*) + M(t^*,1+t^*).
\end{equation}
 
The two norms are coming from the fact that we consider the Brownian increments on two separate intervals $[1,t^*], [t^*,1+t^*]$. For the first term we employ directly \eqref{eq:long_term_bound_wh} and the last bound on $t_e^{-1}$ in \eqref{eq:ca_tstar}:

\begin{equation}\label{eq:onetot_first}
\sup_{z \leq 1} \sup_{t < t^*}\; \frac{ \absv{ \int_1^t (t+z-r)^{H-1/2} dB_r } }{ (1+t)^{H+\delta} } \leq 2 C_H M(1,t^*) .
\end{equation}

For the remaining part, that is for $t \in [t^*, 1+t^*]$ one applies exactly \eqref{eq:ibp} because the length of the interval is equal to one. We then have:

\begin{equation}\label{eq:onetot_second}
\sup_{z \in (0,1)}\absv{ \int_{t^*}^{t^*+z} (t^*+z-r)^{H-1/2} dB_r } \leq c_H M(t^*,1+t^*) .
\end{equation}

 
 
Thus we proved \eqref{eq:rl_bound}. To finish the proof of the estimate \eqref{eq:worst_case}, we need to get a bound on $P^{(-\infty,0)}_t$. By the choice of $c_{A,f}$ done in Definition \ref{def:stage2_constants}:

\begin{equation}\label{eq:past_lower_bd}
\begin{aligned}
P^{(-\infty,0)}_t \leq & c_{A,f} (1+t)^{1} + c_{A,c} (1+t)^{H+\delta}  \leq U^{-\theta/2} (1+t^*)^{H+\delta-1} (1+t)^{1} + c_{A,c} (1+t)^{H+\delta} \\
\leq & 2U^{-\vartheta/2} (1+t)^{H+\delta},
\end{aligned}
\end{equation}

where the last inequality comes from the fact that the first term then is smaller than $U^{-\theta/2} (1+t)^{H+\delta}$ for all $t < t^*$. Before proceeding further we can finally specify the event $A_w$:

\begin{equation}\label{eq:aw_def}
A_w = \{ M(1,t^*) \leq K_A U^{-\vartheta(1/2+\delta)} / 4, M(t^*,1+t^*) < K_A \} .
\end{equation}
 
Note that we always have $K_A < 1$, so in what follows this dependence is going to be ignored. Then putting \eqref{eq:bridge_bound}, \eqref{eq:rl_bound} and \eqref{eq:past_lower_bd} into \eqref{eq:worst_case} one obtains that for some constant $C > 1$ that depends only on $H, \delta$.
\begin{equation}\label{eq:wh_lower_bound}
\wh_t > 12 h(t-1) - C (1+t)^{H+\delta} U^{-\vartheta/2} / 4- 1,
\end{equation}
 
where $h(s) = s^{H-1/2} \lor 1$ for $H >1/2$ and $s^{H-1/2}$ for $H<1/2$ (see Lemma \ref{lem:new_noise_bridge_decomp}).  Recall $Z \geq 12t_e^{-1/2-\delta}$ and $t^* = \left( U t_e^{-1/2-\delta} \right)^{\vartheta}$. 

To finish the proof we have to show that \eqref{eq:wh_lower_bound} is bigger than one. First we will get rid of the constant $-1$ we need $Z h(t-1) - 1 > Z h(t-1) / 2$. For $H>1/2$  we have $\inf_{t \in [1,1+t^*]} h(t-1) > 1$ and $Z > 12 t_e^{-1/2-\delta}$, which is enough. If $H < 1/2$, we get our claim by considering $\inf_t h(t-1) = (t^*)^{H-1/2} = t_e^{-(1/2+\delta)\vartheta(H-1/2)} U^{\vartheta(H-1/2)}$ and applying the assumption \eqref{eq:vartheta_lower} we get the claim. With \eqref{eq:wh_lower_bound} simplified we can set:

\[ g(t) = 6t_e^{-1/2-\delta} h(t-1) /2 - C (1+t)^{H+\delta} U^{-\vartheta/2} ,\]

and then require $g(t) > 1$. This clearly holds for $t < 2$, because then $6 t_e^{-1/2-\delta} - C 2^{H+\delta} U^{-\vartheta/2} > 1$. For $t>2$ it suffices to check:

\begin{equation}\label{eq:min_g_lower_bd}
1 - t_e^{1/2+\delta} C (1+t)^{H+\delta} t^{1/2-H} U^{-\vartheta(1/2+\delta)}  / 24 > c > 0.
\end{equation}

By inserting the definition of $t^*$ and the assumption in $\delta$ given by \eqref{eq:ca_tstar} we have:

\begin{align*}
 \sup_{t < t^*} Z^{-1} C (1+t)^{H+\delta} t^{1/2-H} U^{-\vartheta(1/2+\delta)}  / 4 & \leq 2  t_e^{1/2+\delta} U^{-\vartheta(1/2+\delta)} t_e^{-(1/2+\delta)^2\vartheta} U^{\vartheta(1/2+\delta)} (t^*)^{\delta} / 24 \leq  \\
 & \leq 2 t_e^{1/2 - \vartheta/2 - \delta} / 4 .
 \end{align*}
 
 As $\vartheta < 2$ and $t_e^{-\delta} < 2$ (\eqref{eq:ca_tstar}), the claim \eqref{eq:min_g_lower_bd} is readily proven. This implies:
 
 \[ \inf_{t \in [1,1+t^*]} \wh_t > Z h(t-1) /2 > t_e^{-1/2-\delta} (t^*)^{H-1/2} = t_e^{-1/2-\delta + \vartheta(H-1/2)} U^{\vartheta(H-1/2)} = g(t^*). \]
 
 By \eqref{eq:teu} this is bigger than one and as a consequence one can rewrite \eqref{eq:xt_lower_bd} as:
 
 \[ \blue{\forall t \in [1,1+t^*]}\;\; X_t > g(t^*) - t_e^H + \int_1^t  { f(X_r) } dr.\]
 
The first two terms are $M^*$ in \eqref{eq:vartheta_lower} and their sum is bigger than one. \blue{The last integral is (strictly) positive as long as $X >0$ on $[1,t]$. We then conclude by an argument by contradiction : if $X$ does not stay above $1$ on $ [1,1+t^*]$, let $t$ be the first time for which $X_t=1$, the above estimate implies $X_t > 1$, a contradiction.} 


\paragraph{\textbf{The proof of \eqref{eq:decay_new_noise}}}

We want to get bound on the past process in terms of $M(1,t^*)$ and $M(t^*,1+t^*)$ imposed in \eqref{eq:rl_bound}. To this end, we bound the norms from Lemma \ref{lem:recent_past}. One observes that on the event $A_w$ defined by \eqref{eq:aw_def}:

\[ L(1,t^*) = \sup_{u\in[1,t^*]} \frac{ \absv{ B_{t^*} - B_u } }{ (1+t^*- u)^{1/2+\delta} } \leq (t^*)^{3\delta} M(1,t^*) \leq (t^*)^{3\delta} U^{-\vartheta/2} K_A / 4. \] 

Therefore, inserting this bound to the ones in Lemma \ref{lem:recent_past} one has:

\begin{align*}
\sup_{u,s>0} u^{-1} (1\lor s)^{1-H} \absv{ P^{(1,t^*),1+t^*+s}_u} & \leq (t^*)^{3\delta} K_A U^{-\vartheta/2} / 4  \\ \sup_{u>0} u^{-H+\delta} P^{(t^*,1+t^*),1+t^*+s}_u & \leq K_A.
\end{align*}

Along with estimates from the beginning of the Lemma, this finishes the proof of admissibility. Note that since, in the definition \eqref{eq:aw_def} of $A_w$ the constants in  are fixed, and the two norms are independent (and have the same distribution as $\norm{ B }_{\ccc{1/2-\delta};[0,1]}$ via scaling in Lemma \ref{lem:subgaussian_scaling}), it holds that for some $\lambda_w$ one has $\PP^{\FF_1{}}\left( A_w \right) > \lambda_w$.
\end{proof}

\begin{remark}\label{rem:vartheta_inftwo}
Note that the constraint $\vartheta < 2$ is very natural - we pick any $\beta$ such that $\beta > \inv{\vartheta} + H - 1/2$. If we could take $\vartheta > 2$, this would mean that the exponent in the bound on the growth of $s \mapsto P^{(0,1)}_s$ is smaller than $H$, which is absurd.
\end{remark}

\subsection{Stage 3}\label{subsec:stage_three}

We state the main results of the Section in the following proposition. We take $q$ as in \eqref{eq:fixed_five},  {namely
\begin{equation*}\label{eq:fixed_five}
 \begin{cases} q \in (2^{\inv{\alpha}}, 5^{\inv{\alpha}} ) & \gamma < 0 \\ q \in (1,1+\inv{2K_{\gamma}}) & \gamma > 0 \end{cases}
 \end{equation*}
}and define the sequence of geometric progression

\begin{equation}\label{eq:geometric_times}
T_k =  \sum_{j =1}^k q^k.
\end{equation}


 {Recall as well that $K_A$ has been defined in Definition \ref{def:stage2_constants} in \eqref{eq:fixed_four}}.

\begin{proposition}\label{prop:stage3_prop}
Let $\rho$ be a stopping time such that, a.s.:

\[ P^{(-\infty,\rho)}_{\cdot} \in \mathbf{A}(K_A, K_A) \mbox{ and } \quad \absv{ X_{\rho} } > 1. \]

Let $X_{\rho} > 0$ (respectively $<0$). Let $\alpha$ be as fixed in Theorem \ref{thm:main}. Then define $\tau$ as

\[ \tau = \inf\{ t >0 : \quad X_{\rho+t} < \varphi(1,A^+ t) - t^{\alpha} \} \]

(respectively $X_{\rho+t} > \varphi(-1, A^- t) + t^{\alpha}$ if $X_{\rho} < 0$). Then there exists a $\NN$-valued random variable $k_f$ such that $\tau \geq \rho + T_{k_f}$ and constants $\lambda_3,b, C > 0$ such that:

\[  \PP^{\FF_{\rho}}\left( k_f = \infty \right) \geq \lambda_3, \; \quad \EE^{\FF_{\rho}} \left[ \exp\left( b T_{k_f}^{2(\alpha-H) } \right) \indic{k_f < \infty}\right] \leq 1 + b C  .\]

\end{proposition}

\begin{proof}
Again we take $\rho=0$ to simplify the notation below. We approximate $\tau$ by the sequence of geometric times \eqref{eq:geometric_times}. First by Lemma \ref{lem:curve_gam_neg} and Lemma \ref{lem:curve_gam_pos}, in which deterministic conditions on the noise are given - for $k_f \in \NN$ as defined there one observes that $\tau \geq T_{k_f}$.

Then Lemma \ref{lem:infinite_brownian_holder_norms} specifies a growth condition of independent Brownian H\"older norms under which $k_f = \infty$, and in Lemma \ref{lem:proba_negative} we show that there exists $\lambda > 0$ such that $\PP\left( \tau = \infty \right) > \lambda$. The tail estimate in the case of $\tau < \infty$ is shown in Lemma \ref{lem:climbing_tails}.  {The choice of $\alpha$ is important in Lemma \ref{lem:infinite_brownian_holder_norms} }
\end{proof}

We now start with a deterministic bound on the dynamics. We separate the case of $\gamma < 0$ and $\gamma \geq 0$, which have different monotonicity properties for the drift (which is why in the case $\gamma <0$ we cannot apply directly the methods of \cite{DF14, Tre13}).  
\begin{lemma}\label{lem:comparison}
Let  {$f \in \ccc{1}(\RR)$ be non-increasing } and let $\varphi$ be the semi-flow associated to the ODE  { $\dot{z} = f(z)$ }. Moreover, let $(w_t)_{t \geq 0}$ be a real-valued path. Consider the following deterministic ODE, for a given $X_s > 0$ :

\begin{equation}\label{eq:ode_deterministic}
\forall u \in [s,t]\;\; X_u = X_s + \int_s^u  { f(X_r) } dr + w_u - w_s . 
\end{equation}

Assume that there exists a constant$\bar{w}$ such that: 

\[ \sup_{r \in [s,t]} \absv{ w_r - w_s } \leq \bar{w}. \] Then we have the bounds:

\begin{equation}\label{eq:finite_horizon_decreasing}
\forall u \in [s,t]\;\; \varphi(X_s + \bar{w},  u-s) - 2 \bar{w}  \leq X_u \leq \varphi(X_s - \bar{w}, u-s ) + 2 \bar{w}.
\end{equation}

\end{lemma}

\begin{proof}
Without loss of generality assume that $s = 0, t = T$ and $w_0 = 0$. We only show the first bound, the rest being identical up to switching the signs and inequalities where necessary. Let us rewrite the equation as:

\[ \forall t \leq T, \;\; \theta_t = X_t - w_t =  \int_0^t f( \theta_r + w_r) dr.   \]

We will compare to the following ODE, which we can solve explicitly:

\begin{equation}\label{eq:auxiliary_ode}
\bar{\theta}_t = \blue{ \theta_0} + \int_0^t f( \bar{\theta}_r + \bar{w} ) dr , \;\; \bar{\theta}_0 = \theta_0 \blue{=X_0}.
\end{equation}

 \blue{It is easy to check that $\bar{\theta}_t = \varphi( \theta_0 + \bar{w}, t ) - \bar{w}$. 
 It holds that $\pdt \theta_t = f(\theta_t + w_t) = f_t(\theta_t)$ with $f_t( \cdot ) = f(\cdot + w_t)$. Then by the monotonicity of $f$ there holds $f_t(x) \geq f(x + \bar{w})$ and by comparison principle for one dimensional ODEs we have that $\theta_t \geq \bar{\theta}_t= \varphi( \theta_0 + \bar{w}, t) - \bar{w}$. We then finish by observing that}

\blue{
\[ X_t = \bar{\theta}_t + \underbrace{ \theta_t - \bar{\theta}_t }_{\geq 0}  + w_t \geq \bar{\theta}_t - \bar{w}. \]
}
 \end{proof}

 {
\begin{remark}
Assume that $f$ is only $C^1$ (and non-increasing) on a sub-interval $I \subset \RR$, and that $X$ is continuous and such that, for any interval $[u,v] \subset [s,t]$,
\[  X_r \in I\; \forall r \in [u,v] \;\;\Rightarrow \;\;\;X_u = X_v + \int_v^u f(X_r) dr +  w_u - w_v, \]
and that in addition both the r.h.s. and l.h.s. of \eqref{eq:finite_horizon_decreasing} are in $I$. Then it follows by a simple localization argument and an application of Lemma \ref{eq:ode_deterministic}, that \eqref{eq:finite_horizon_decreasing} still holds in this case.

We will use this remark repeatedly below, taking $I = (0,\infty)$ or $I=(-\infty,0)$ and $f(x) = sign(x) |x|^{\gamma}$, for $\gamma<0$.
\end{remark}
}
Recall that $\varphi=\varphi(x,h)$ is the (semi-)flow associated to $\dot{z}=z^{\gamma}$ i.e. for $x>0$, $ \varphi(x,0)=x, \partial_h \varphi = \varphi^{\gamma}$.

\begin{lemma}\label{lem:curve_gam_neg}
Assume that $\gamma < 0$, and let $X_0 \geq 1$. Fix $\alpha < \inv{1-\gamma}, k_f \in \NN \cup \{\infty\}$, let $T_k$ be the sequence set in \eqref{eq:geometric_times} with $q \in (2^{\inv{\alpha}}, 5^{\inv{\alpha}})$.
Let the path $(w_t)_{t \in [0,T_{k_f}]}$ satisfy:

\begin{equation}\label{eq:path_w_growth}
\forall k \in \{0,... k_f \}\;\; \sup_{r < T_{k+1} - T_k} \absv{ w_{r + T_k} - w_{T_k}} \leq c_w q^{k \alpha}, \quad   \mbox{ for some }c_w \leq 3 K_{A^+},
\end{equation}

with  $K_{A^+}$ is as in  {Definition \ref{def:stage2_constants}}. 
Then:

\[ \forall u \in [0, T_{k_f}]\;\; X_u \geq \varphi(1, A^+ u)  - 2c_w q^{\alpha} q^{k(u) \alpha},  \mbox{ where } k(u) = \sup\{ k \in \NN \cup \{0\} : T_k < u \}. \]

and with this choice of $c_w$ we have $\forall u \in [0, T_{k_f}]\;X_u \geq 0$. Analogously, for $X_0 < -1$ there holds:

\[ \forall u\in [0,T_{k_f}]\;\; X_u \leq -\varphi(1, A^- u ) + 2c_w q^{\alpha} q^{k(u) \alpha} , \]

with $c_w \leq 3K_{A^-}$ instead of $3K_{A^+}$.
\end{lemma}

\begin{proof}
We will only prove the case of positive $X$, the negative one follows with the same arguments. We will employ Lemma \ref{lem:comparison} with $ { f(x) } = x^{\gamma}$, which is differentiable and decreasing for $x>0$. (Note that the bounds we obtain imply that our solution stays bounded from below by a $x_0 >0$). We will show by induction that, for all $1 \leq k \leq k_f$,

\begin{equation}\label{eq:xk_assumption}
 {\forall u \in [0, T_k - T_{k-1}]\;\;X_{T_{k-1}+u} \geq \varphi\left( X_0, A^+ (T_{k-1}+u) \right) - 2 c_w q^{\alpha k}},
\end{equation}

 { Let us start with the base case of the induction.} The case $k=1$ follows by applying Lemma \ref{lem:comparison} with interval $[0,T_1], \bar{w} = c_w q^{\alpha}$ where $K_A, q$ are fixed such that $\bar{w} \leq 2K_A q^{\alpha} < 1$. Then, for $u \in [0, T_1]$

\begin{equation}\label{eq:ubound_first_interval}
X_{u} \geq \varphi( X_0 + c_w q^{\alpha}, A^+ u ) - 2 c_w q^{\alpha} .
\end{equation}

At the end of the interval we have:

\[ X_{T_1} \geq \varphi(X_0 + c_w q^{\alpha}, A^+ T_1 ) - 2 c_w q^{\alpha} .\]

 {We can now proceed with the induction step. This means that we will} show that if \eqref{eq:xk_assumption} holds  {for some $k \geq 1$, then}

\[ \forall u \in [0, T_{k+1}-T_k]\;\;X_{T_k+u} > \varphi\left( X_0, A^+ \left( T_k + u \right) \right) - 2 c_w q^{\alpha(k+1)}, \]

which for $u = T_{k+1} - T_k$ is exactly the bound \eqref{eq:xk_assumption} with $k+1$ instead of $k$. 

We apply Lemma \ref{lem:comparison} with $\bar{w}^k := c_w q^{\alpha(k+1)}, [S,T] = [T_k, T_{k+1}]$.  {Let us for now assume that with this choice of $c_w$ the Lemma \ref{lem:comparison} is indeed applicable}, the fact that this choice of $c_w$ will guarantee that the path will stay positive on this interval will be shown at the end of the proof.  {Indeed, under this assumption we can write}  
\begin{equation}\label{eq:lower_bound_xtk}
\begin{aligned}
X_{T_k+u} \geq \; & \varphi\left(X_{T_k} + c_w q^{\alpha(k+1)}, A^+ u \right) - 2 c_w q^{\alpha(k+1)} \\
\geq \; & \varphi\left( \varphi\left( X_0, A^+ T_k \right) - 2 c_w q^{\alpha k} + c_w q^{\alpha(k+1)}, A^+ u \right) - 2c_w q^{\alpha(k+1)} \\
\geq \; & \varphi\left( X_0, A^+ \left( T_k + u \right) \right) - 2 c_w q^{\alpha(k+1)}  ,
\end{aligned}
\end{equation}
where we have used Lemma \ref{lem:comparison} in the first line, \eqref{eq:xk_assumption} in the second line, and in the third line the properties of the flow: monotonicity in $x$ for $(x,t) \mapsto \varphi(x,t)$ and $\varphi(x,t+s) = \varphi(\varphi(x,s),t)$.  {We conclude the induction step by taking $u \in [0, T_{k+1} - T_k]$ as announced}.

\smallskip

 We now have to check that $X_u \geq 0$,  {therefore justifying the first line of \eqref{eq:lower_bound_xtk}}. Defining $k(u)$ as in the statement of the lemma, one can then write:

\[ X_u \geq \varphi( X_0, A^+ u ) - 2 c_w q^{\alpha(k(u)+1)} \geq \varphi( X_0, A^+ u ) \left( 1 - 2 c_w q^{\alpha} q^{\alpha k(u)} \varphi\left(X_0, A^+ T_{k(u)} \right)^{-1} \right) .\]

Setting $h(t) := 1 - 2 c_w t^{\alpha} \varphi\left(X_0, A^+ t \right)^{-1}$, our goal is to now show that with the assumed choice of $q$ and $c_w \leq 3K_A$ with $K_A$ fixed in Definition \ref{def:stage2_constants} one has $\inf_{t > 0} h(t) > 0$. With Lemma \ref{lem:max_deviation} we see that with our choice of $c_w$ the right hand-side is always positive. Therefore the choice of:

\[ c_w \leq (A^+)^{\alpha} \left( X_0 + c_{\gamma,\alpha} \right)^{\inv{1-\gamma}} \]

leads to $X_u \geq 0$ as announced. This right hand-side of this bound is not bigger than the assumed $K_A$. The exact same computation holds for $A^-$ instead of $A^+$.  {Since $X_u \geq 0$, every use of Lemma \ref{lem:comparison} is valid and the induction step is justified, therefore showing \eqref{eq:xk_assumption} for every $k \in [1, k_f ] \cap \NN$ as announced and therefore the claim of the Lemma.}
\end{proof}

We now obtain the corresponding results in the case $\gamma>0$.The following lemma is a modified version of Lemma 2.2 in \cite{Tre13}.

\begin{lemma}\label{lem:positive_comparison}
Let  {$f \in \ccc{1}(\RR)$ be non-decreasing } and again let $\varphi$ be the flow associated to  {$\dot{z} = f(z)$}.  Let $(w_u)_{u \in [s,t]}$ be a continuous path and $X$ satisfy
\begin{equation*}\label{eq:ode_deterministic2}
\forall u \in [s,t]\;\; X_u = X_s + \int_s^u  { f(X_r) } dr + w_u - w_s . 
\end{equation*}
Then the following bounds hold

\begin{equation}\label{eq:finite_horizon_lower_bd}
\forall u \in [s,t],\;\; \varphi\left( X_s - \bar{w}, u-s \right) \leq X_u \leq \varphi\left( X_s + \bar{w}, u-s \right) .
\end{equation}

\end{lemma}

\begin{proof}
  Let $\theta = X - w$. Then \[ \theta_t = x_0 + \int_0^t b( \theta_r + w_r - w_{t_0}) dr \]

and $\bar{\theta}_t = x_0 + \int_0^t f( \bar{\theta}_r - \bar{w} )dr, \bar{\theta}_0 > \bar{w}$. One observe that $\bar{\theta}_t = \varphi( x_0 - \bar{w}, A^+ t ) + \bar{w}$. Setting $\rho = \theta - \bar{\theta}$ we get:

\begin{equation}\label{eq:difference_ode_gamma_pos}
\rho =  \int_0^t f( \rho_r + \bar{\theta}_r + w_r - w_{t_0} ) - f( \bar{\theta_r} - \bar{w} ) dr .
\end{equation}

There holds $\pdt \rho > 0$ for any $w_r - w_{t_0} > - \bar{w}$, which is true by definition of $\bar{w}$. The claim follows by inserting the solution for $\bar{\theta}$.
\end{proof}

The following lemma is the counterpart of Lemma \ref{lem:curve_gam_neg} for $\gamma > 0$.  { Similarly to Lemma \ref{lem:curve_gam_neg} the flow $\varphi(x,h)$ means (semi-)flow of the ODE $\dot{z} = z^{\gamma}$. } 
\begin{lemma}\label{lem:curve_gam_pos}
Let $X_0 > 1$ and $\gamma > 0$ Fix $k_f \in \NN$ (possibly $k_f = \infty$) and let $q$ be as fixed in \eqref{eq:fixed_five}. Let the path $(w_t)_{t \in [0,T_{k_f}]}$ satisfy:

\begin{equation}\label{eq:path_w_growth_pos}
\forall k \in \{1, ... k_f \}, \;\; \sup_{r < T_{k+1} - T_k} \absv{ w_{r + T_k } - w_{T_k}} < c_w q^{k \alpha}, \quad \mbox{ for some }c_w \leq 3K_A.
\end{equation}

Then:

\[ \forall u \in [0, T_{k_f}],\;\; X_u > \varphi(1, A^+ u)  - c_w q^{\alpha} u^{\alpha}  .\]

\end{lemma}

\begin{proof}
Similarly to the case of $\gamma < 0$, we get that the bound is satisfied in $[0,T_1]$ because $c_w q^{\alpha} \leq q^{\alpha} 3K_A \leq 1$ due to the value of $K_A$ fixed in Definition \ref{def:stage2_constants}.
To show the bound for the next steps  {we will first prove that } there exists a constant $K > 0$, that will be shown to depend only on $\gamma$ such that for any $\bar{w} > 0$ there holds:

\begin{equation}\label{eq:flow_bound}
\forall k\in\NN\; \forall t < q^{k+1} \;\;\varphi\left( \varphi\left(0, T_k\right)- \bar{w},A t \right) > \varphi\left(0 ,A \left( T_k +t \right) \right) - K A^{-1} \bar{w} (q-1)^{1-\gamma} q^{\gamma} .
\end{equation}

Let $C > 1$ be some constant that we will precise later. We claim that if for $k \in \NN$ there holds:

\begin{equation}\label{eq:gampos_induction_assumption}
X_{T_k} > \varphi(0, A T_k) - C q^{k\alpha} ,
\end{equation}

then the same is true for $k +1$. Indeed, by Lemma \ref{lem:positive_comparison} applied with $ { f(x) } = \sgn{x} \absv{x}^{\gamma}$ for $X_u \geq 0$ and the assumption on the noise and plugging in \eqref{eq:gampos_induction_assumption}:

\[ \forall t < q^k\;X_{T_k+t} > \varphi\left( X_{T_k} - c_w q^{(k+1)\alpha}, t \right) > \varphi\left( \varphi\left(0, A T_k\right) - C q^{k\alpha} - c_w q^{(k+1)\alpha}, t \right). \]

We then use \eqref{eq:flow_bound} with $\bar{w} = C q^{k\alpha}+c_w q^{(k+1)\alpha}$ and obtain:

\[ X_{T_k+t} > \varphi\left(0, A(T_k+t) \right) - \left( C q^{k\alpha} + c_w q^{(k+1)\alpha} \right) K_{\gamma} A^{-1} (q-1)^{1-\gamma} q^{\gamma}. \]

If we manage to show that the second term is smaller up to absolute value than $C q^{(k+1)\alpha}$, then \eqref{eq:gampos_induction_assumption} will hold for $k+1$ as well and the induction will be proven. Indeed, set $f(q) = (q-1)^{1-\gamma} q^{\gamma}$. Constant $K_{\gamma}$ is the constant in the assumption on $q$ and therefore it holds $f(q) < \inv{2K_{\gamma}}$:

\[ \left( C q^{k\alpha} + c_w q^{(k+1)\alpha} \right) \cdot K f(q) A^{-1} = C q^{k\alpha} K f(q) + c_w K q^{(k+1)\alpha} f(q) \leq C q^{(k+1)\alpha}  .\]

In this case it suffices to set $C = c_w$ and we obtain the multiplicative bound in the same way as in the case of $\gamma < 0$.

We can now prove \eqref{eq:flow_bound}. Let us observe that for all $x>1, t> 0$ there holds 

\[ \pdx \varphi(x,t) \leq t x^{\gamma-1} C_{\gamma}, \]

where $C_{\gamma}$ depends only on $\gamma$. Then using the bound $t < q^{k+1}$:

\begin{align*}
\varphi( y_{A T_k} - \bar{w}, A t) & > \varphi( y_{A T_k}, A t ) -  \bar{w} \cdot \pdx \varphi( y_{A T_k} -  \bar{w}, A t ) \\
& > \varphi( y_{A T_k}, t ) -  \bar{w} \cdot A q^{k+1} \cdot \left( y_{A T_k} -  \bar{w} \right)^{\gamma-1} C_{\gamma} \\
& > \varphi( y_{A T_k}, t) -   \bar{w} \cdot A q^{k+1} y_{A T_k}^{\gamma-1} C \\
& > \varphi( y_{A T_k} , t) -  C_{\gamma}  \bar{w} q^{k+1} 2^{1-\gamma} q^{-k} \left( \frac{q}{q-1} \right)^{\gamma-1},  \\
\end{align*}

so that we take $K_{\gamma} = 2 C_{\gamma}$. This implies that $q$ can be chosen only in function of $\gamma$. We check that the assumption of $X_u > 0$ necessary to apply Lemma \ref{lem:positive_comparison} is valid in the same way as in Lemma \ref{lem:curve_gam_neg} - hence the same assumption on $c_w$.
\end{proof}

The final step is to compute the conditions under which the assumptions on the noise specified in Lemmas \ref{lem:curve_gam_pos} and \ref{lem:curve_gam_neg} are satisfied.  

\begin{proposition}\label{prop:increasing_noises}
Let $\beta' > H > \beta, q>1$ and $T_k$ as in \eqref{eq:geometric_times}. Fix $\eta > 0$ and assume that $\forall j \leq k,\;\;\norm{ B }_{\ccc{1/2-\delta};[T_j, T_{j+1}]} \leq c_w q^{(j+1)\eta}$ Then for some other constant $C_{q,\eta} > 0$

\[ \forall s >0\;\int_0^{T_k} G(s, T_k, r) dB_r \leq c_w C_{q,\eta} \left( s^{\beta} q^{k(\eta+H-\beta)} + s^{\beta'} q^{k\eta} \right) .\]

\end{proposition}

\begin{proof}
We decompose the integral into the following sum:

\[ \int_0^{T_k} G(s, T_k, r) dB_r = \sum_{j=0}^{k-1} \int_{T_j}^{T_{j+1}} G(s,T_k,r) dB_r =: \sum_{j=0}^{k-1} I^j .\]

For the last interval we have, applying \eqref{eq:ibp} - here $\beta < H$ is crucial because otherwise the integral will not be finite:

\begin{align*}
I_k \leq & s^{\beta} \absv{T_k - T_{k-1}}^{H-1/2-\beta} \absv{B_{T_k} - B_{T_{k-1}} } + s^{\beta} \absv{T_k - T_{k-1}}^{H-\beta-\delta} \absv{ B }_{\ccc{1/2-\delta};[T_{k-1},T_k]} \\
\leq & 2 s^{\beta} \absv{ T_k - T_{k-1}}^{H-\beta-\delta} \norm{ B }_{\ccc{1/2-\delta};[T_{k-1},T_k]} \leq 2 s^{\beta} c_w q^{k(\eta+H-\beta-\delta)} .
\end{align*}

For $j \in \{1,... k-2\}$ and $\beta' > H$ similar reasoning leads to:

\begin{align*}
I_j \leq & 2 s^{\beta'} \absv{ T_k - T_{j-1} }^{H-\beta-1/2} \absv{ B_{T_j} - B_{T_{j-1}}} \leq 2 s^{\beta'} \absv{ T_k - T_{j-1} }^{H-\beta'-\delta} \norm{ B }_{\ccc{1/2-\delta};[T_{j-1},T_j]} \leq 2 s^{\beta'} c_w q^{j\eta}.
\end{align*}
 
 The constant $C_{q,\eta}$ comes from summing up a geometric progression with coefficient $q^{\eta}$.
\end{proof}

\begin{remark}
Observe that Proposition \ref{prop:increasing_noises} agrees with the usual bounds that we obtain for the behaviour of fBm, for instance in Section 2 \cite{Pic11} - short term continuity is bounded with $s^{\beta}, \beta < H$, whereas long term behaviour is bounded with $s^{\beta'}, \beta' > H$.
\end{remark}

 {
Recall that $K_A$ is fixed to satisfy $K_A \leq \frac{5^{\inv{\alpha(1-\gamma)}} \land (1 + A^{\inv{1-\gamma}})}{2}$, that $\alpha$ satisfies
$\alpha \in (H, \inv{1-\gamma})$, 
that $q$ is fixed in $ \left( 2^{\inv{\alpha}}, 5^{\inv{\alpha}} \right)$, and that $T_k =  \sum_{j =1}^k q^k$.

\begin{lemma}\label{lem:infinite_brownian_holder_norms}
Assume that

 \[ P^{(-\infty,0)}_{\cdot} \in \mathbf{A}(K_A, K_A),  \]
 and for $k \in \{ 1, ..., k_f \}$, where $ k_f \in \NN \cup \{ \infty \}$, the sequence of Brownian noises satisfies \[ \norm{ B }_{\ccc{1/2-\delta};[T_k,T_{k+1}]} < \frac{ K_A }{ c_{H,\alpha}} q^{k (\alpha-H)}, \] where $c_{H,\alpha}$ is some positive constant.

Then the assumptions \eqref{eq:path_w_growth} and \eqref{eq:path_w_growth_pos} on the path $(w_t)_{t \in [0, T_{k_f}]}$ in Lemmas \ref{lem:curve_gam_neg} and \ref{lem:curve_gam_pos} are satisfied with $c_w = 3 K_A$.
\end{lemma}
}
\begin{proof}

With $T_k$ as above, for the term depending on the remote past, 
we employ the definition of the admissible condition (Def \ref{def:admissibility}) with $k \geq 1, t > T_k$ and $s \in (T_k,T_{k+1})$. We decompose every increment in the following way:

\begin{equation}\label{eq:wh_decomp_three_terms}
\wh_t - \wh_{T_k} = \int_{T_k}^t (t-r)^{H-1/2} dB_r + \int_0^{T_k} G(t-T_k, T_k,r) dB_r + P^{(-\infty,0),T_k}_{t-T_k}.
\end{equation}

By the assumption on admissibility of the past we have:

\begin{equation}\label{eq:past_noise_increment_bound}
\forall k \in \NN\;P^{(-\infty,0),T_k}_{t - T_k} \leq K_A (1+T_k)^{ \left( \alpha - 1\right) \land 0 }  \absv{t-T_k}^{1}  + K_A \absv{t-T_k}^{H-\delta} \leq 2K_A  { q^{k\left( \left( (\alpha - 1) \land 0 \right) + 1 \right)}}.
\end{equation}

Note that by  { Theorem \ref{thm:main} }one has $\alpha < 1/(1-\gamma)$. By \eqref{eq:long_term_bound_wh} and Proposition \ref{prop:increasing_noises} and the fact that $c_{H,\alpha}$ can be fixed exactly to cancel the implicit constant we obtain:

\begin{equation}\label{eq:new_noise_increment_bound}
\int_{T_k}^t (t-r)^{H-1/2} dB_r + \int_0^{T_k} G(t-T_k, T_k,r) dB_r \leq K_A q^{k(H+\delta)} T_k^{\alpha-H} .
\end{equation}

By estimating $T_k$ with an obvious bound on geometric progression (which is where the implicit constant is coming from) and putting \eqref{eq:new_noise_increment_bound}, \eqref{eq:past_noise_increment_bound} into \eqref{eq:wh_decomp_three_terms} it then follows that:

\[ \sup_{T_k \leq t \leq T_{k+1}} \absv{ \wh_t - \wh_{T_k} } \leq  { 2K_A q^{k\left( \left( (\alpha - 1) \land 0 \right) + 1 \right)} +  K_A q^{k(\alpha-\delta)} \leq 3 K_A q^{k\alpha} }. \]

We can therefore bound both by $q^{k\alpha}$ with parameter $\alpha$ as chosen in Definition \ref{def:stage2_constants} and in restated the statement of the Lemma. 
\end{proof}

We can compute the probability of sampling an infinite sequence of suitable Brownian increments. The choice of $\eta$ is irrelevant for the proof, the intended use is $\eta = \alpha - H$.
\begin{lemma}\label{lem:proba_negative}
Let $\eta > 0$. Fix $c_w \in (0,1), q > 1$. Then there exists $\lambda_{\eta} > 0$ s.t. :

\[ \PP^{\FF_0}\left( \forall j \geq 0\;\; \norm{ B }_{\ccc{1/2-\delta};[T_j,T_{j+1}]}  < c_w q^{j\eta} \right) > \lambda_{\eta}. \]

\end{lemma}

\begin{proof}
Thanks to Lemma \ref{lem:subgaussian_scaling}. $\EE^{\FF_0} \left[ \exp\left( \lambda_0 \left( \inv{ \absv{ T_{j+1} - T_j }^{\delta} } \norm{ B }_{\ccc{1/2-\delta};[T_j,T_{j+1}]} \right)^2 \right) \right] < 2$ for a $\lambda_0>0$ which can be picked uniformly in all $j \geq 1$. Therefore there exists $j_0 \in \NN$ such that for all $j \geq j_0$ there holds

 {
\[ \PP^{\FF_0}\left( \norm{ B }_{\ccc{1/2-\delta};[T_j,T_{j+1}]} < x \right) > 1 - 2 \exp\left( - \lambda_0 \left(  \inv{\absv{T_{j+1} - T_j}^{\delta}} x \right)^2 \right)  .\]
}

In our case $x = c_w q^{j \eta}$ and $\absv{ T_{j+1} - T_j } = q^j$. This implies that the above can be bounded with 
\begin{equation}\label{eq:fernique_lower_bd}
 p^{b,\eta}_j >  1 - 2 \exp\left( - \lambda_0 c_w^2 \left( q^{ j(\eta - \delta_0) } \right)^2 \right) ,
 \end{equation}

and a computation employing sum of geometric series shows that $\prod_{j=j_0}^{\infty} p_j^{b,\eta} > 0$.

For $j \leq j_0$, 
it holds by scaling that for positive $p^{s,\eta}_j$ 

\[ \PP\left( \forall j < j_0\;\; \norm{ B }_{\ccc{1/2-\delta};[T_j,T_{j+1}]} \geq c_w q^{j\eta} \right) > \prod_{j=1}^{j_0} \PP\left( \norm{ B }_{\ccc{1/2-\delta};[0,1]} < c_w q^{j(\eta-\delta)} \land 1 \right) \geq \prod_{j=1}^{j_0} p^{s,\eta}_j .\]

As $j_0$ is finite, the right hand-side is strictly positive, so we can define $\lambda_{\eta}$ as follows:

\[ \lambda_{\eta} = \prod_{j=0}^{j_0} p^{s,\eta} \prod_{j=j_0}^{\infty} p^{b,\eta} .\]
\end{proof}

Similarly we estimate the tails of the random variable $T_{k_f}$, where $k_f$ is the final index for which we managed to sample the noise under control.

\begin{lemma}\label{lem:climbing_tails}
Let $\kappa < 2 \eta$. Define the random variable:

\[ k_f = \sup\{ k \in \NN: \norm{ B }_{\ccc{1/2-\delta};[T_j,T_{j+1}]} < q^{k\eta} \}. \]

Then there exist  $a , C > 0$ such that :

 \[ \EE \left[ \exp\left( a \left(T_{k_f} \indic{ k_f < \infty} \right)^{\kappa} \right) \right] < 1 + a C. \]
\end{lemma}

\begin{proof}
For brevity of notation set $U = T_{k_f} \indic{ k_f < \infty}$ (the random variable $U$ is defined only locally in the proof of this Lemma). We will aim at using Taylor expansion, so that first we get a bound on  the moments of  { $U^{\kappa}$ }:

\[ \EE U^{\kappa m} = \EE \left[ U^{\kappa m} \indic{ U < q} \right] + \sum_{k \geq 1} \EE \left[ U^{\kappa m} \indic{ k_f = k } \right] . \]

The first term is bounded by $q^{\kappa m}$. For the second one we first recall that by Lemma \ref{lem:subgaussian_scaling} there holds:

\[ \PP\left( k_f = k \right) \leq \PP\left( \norm{ B }_{\ccc{1/2-\delta};[T_{k+1},T_{k+2}]} > q^{(k+1)\eta} \right) \ls  e^{- q^{2(k+1)(\eta-\delta)} }. \]

Summing geometric progression we get $T_k \ls q^{k}$, so that the second one can be estimated by comparing a series with an integral as follows:

\[ \sum_{k \geq 1} q^{\kappa m(k+1)} \exp\left( - q^{2(\eta-\delta) k } \right) \ls q^{\kappa m} \int_{\RR_+} x^{(\kappa m) } e^{-x^{2(\eta-\delta)}} dx = u^{\kappa m } \Gamma\left( \frac{\kappa m + 1}{2(\eta-\delta)} \right), \]

where $\Gamma$ is a Gamma function. Using Stirling approximation and $\kappa < 2(\eta - \delta)$ we get:

\[ \inv{m!} \EE \left[ a \left( U^{\kappa} \right) \right]^m \leq a^m \left( q^{\kappa m} + q^{\kappa m} e^m \right) . \]

Therefore for small enough $a > 0$ there exists a constant $C > 0$ such that $\EE \exp\left( a U^{\kappa} \right) < 1 + aC$.
\end{proof}

We give here the proof of a technical lemma that leads to the choice of $c_w$.
\begin{lemma}\label{lem:max_deviation}
Let $h(t) = c_w t^{\alpha} (M+At)^{-\inv{1-\gamma}}, t > 1$. Set $m_{\alpha,\gamma} = \frac{\alpha(1-\gamma)}{1+\alpha(\gamma-1)}$. Then:

\begin{enumerate}
\item $h$ reaches its maximum at $t_0 = m_{\alpha,\gamma} \frac{M}{A}$ ,
\item $ h(t_0) = m_{\alpha,\gamma}^{\alpha} c_w  A^{-\alpha} ( M + m_{\alpha,\gamma} )^{-\inv{1-\gamma}} $ .
\end{enumerate}

\end{lemma}

\begin{proof}
For the assumed values of $t$ the function $h(t)$ is differentiable. Then:

\begin{align*}
h'(t) = & c_w \alpha t^{\alpha-1} (M+At)^{-\inv{1-\gamma}} - \inv{1-\gamma} c_w t^{\alpha} (M+At)^{-\inv{1-\gamma}-1} A \\
= & c_w t^{\alpha-1} (M+At)^{-\inv{1-\gamma}-1} \left( \alpha (M+At) - \inv{1-\gamma} t A \right) = c_w t^{\alpha-1} (M+At)^{-\inv{1-\gamma}-1} \left( \alpha M + A t(\alpha-\inv{1-\gamma}) \right).
\end{align*}

Therefore $h(t)$ reaches its maximum at $t_0 = \frac{ M \alpha(1-\gamma)}{1+\alpha(\gamma-1)} \inv{A} =: m_{\alpha,\gamma} \frac{M}{A}$ and the conclusion follows.
\end{proof}

\section{Proof of the main theorem}\label{sec:proof_main}

Before proving Theorem \ref{thm:main}, let us show tail estimates on the  {$\Delta_j$, which are the crucial building blocks in the construction of $\rho_{k^*}$.}

\begin{proposition}\label{prop:tails_of_waiting_time}
 {For $\Delta_j$ as in Definition \ref{def:deltas}} there exist constants $b,M > 0$ such that there holds:

\[ \forall k \in \NN,\;\;\; \EE \left[ \exp\left( b \left(\sum_{j=1}^k \Delta_j \indic{ \Delta_j < \infty} \right)^{\kappa} \right) \right] \leq \left( 1 + bM \right)^k. \]

\end{proposition}

\begin{proof}
For brevity we will denote $\eta = \alpha - H$.  Adding up \eqref{eq:delta1_def}, \eqref{eq:delta2_def} and \eqref{eq:delta3_def} we get the following inequality on $\Delta_j$:

\begin{equation}\label{eq:worst_case_deltaj}
\Delta_j \indic{\Delta_j < \infty} \leq C_W j^{ {\inv{\kappa}-1}} + L(\rho_j - \Delta_{j-1}, \rho_j)^{ { \frac{2}{\kappa }}} + T_{k_f^j} \indic{k_f^j<\infty} +1+ t^* ,
\end{equation}

where $k_f^j$ is the number of successful attempts at sampling correct noise in Lemma \ref{lem:infinite_brownian_holder_norms}. We recall that the $L$-norm above is stochastically dominated by the random variable $V(\infty)$ in the proof of Proposition \ref{prop:gaussian_tails_norms}, which has Gaussian tails. By Proposition \ref{prop:stage3_prop} the random variable $T_{k_f^j}^{2(\alpha-H)}$ has exponential tails. Note that by the choice of $\kappa, \alpha$ we have $\kappa < \frac{2}{3}(\alpha-H) < 2\eta$. Then by Lemma \ref{lem:climbing_tails} we have that $T_{k_f^j}$ satisfies for some $b,M_T>0$:

\begin{equation}\label{eq:tkf_bound}
\forall j \in \NN\;\; \EE^{\FF_{\rho_j}} \left[ \exp\left( b T_{k^j_f}^{\kappa} \right) \indic{ k^j_f < \infty } \right] \leq 1 + b M_T .
\end{equation}

For brevity we denote $L_j := L(\rho_j - \Delta_{j-1}, \rho_j)$. Using concavity of $x \mapsto x^{\kappa}$ for $\kappa \in (0, 1)$ in the first inequality and Cauchy-Schwarz in the second one we obtain:

\begin{equation}\label{eq:rho_k_bound}
\begin{aligned}
\EE \exp\left( b \rho_k^{\kappa} \right) \leq & \exp\left( b \left( \sum_{j=1}^k j^{\inv{\kappa}-1} \right)^{\kappa} \right) \EE \left[ \exp\left( b \sum_{j=1}^k L_j^{2} + b \sum_{j=1}^k T_{k_f^j}^{\kappa} \indic{ k_f^j < \infty} \right) \right] \\
\leq & \exp\left( b k^{\kappa(\inv{\kappa})} \right) \left( \EE \left[ \exp\left( 2 b \sum_{j=1}^k L_j^{2} \right) \right] \right)^{1/2} \left( \EE \left[ \exp\left(2 b \sum_{j=1}^k T_{k_f^j}^{\kappa} \indic{ k_f^j < \infty} \right) \right] \right)^{1/2} .
\end{aligned}
\end{equation}

We note that by standard exponential approximations there exists $M_c > 0$ such that $e^{b k } \leq (1+ bM_c)^k$. The first expectation is a product of quantities that depend only on independent Brownian increments. Therefore by Gaussian tails of $L_j$ there exists $M_L$ such that:

\[ \EE \left[ \exp\left( 2 b \sum_{j=1}^k L_j^2 \right) \right] \leq (1+bM_L)^k .\] 

By \eqref{eq:tkf_bound} we get analogous bound on the second expectation - we condition successively starting with $\FF_{\rho_k}$. Therefore \eqref{eq:rho_k_bound} can be written as:

\[ \EE \exp\left( b \rho_k^{\kappa} \right) \leq \left( 1+b M_c \right)^k \left( 1 + 2b M_L \right)^{k/2} \left( 1 +2 b M_T \right)^{k/2}. \]

This can be further bounded by $(1+bM)^k$ for some $M > 0$.
\end{proof}

The following proposition is \textsl{de facto} punchline of the paper - here is where the computations from the sections above come together, which shows that the last time at which the path can drop below the extremal curve is almost surely finite and has suitable decay.

\begin{proposition}\label{prop:tails_psieta}
 {Let $\alpha, \kappa$ be as in Theorem \ref{thm:main} and $\psi_{\alpha,\kappa}$ be the time defined in \eqref{eq:psi_definition}. Then there holds:}

\[ \EE \left[ \exp\left( b \psi_{\alpha,\kappa}^{\kappa} \right) \right] < \infty. \]

\end{proposition}

\begin{proof}
It follows by definition of $\Delta_j$ that \[\PP\left( \Delta_j = \infty \right) > \lambda_g \cdot \lambda_s =: \lambda. \] Recall that by \eqref{eq:psi_definition} we have $\psi_{\alpha,\kappa} = \rho_{k^*} + 1 + t^*$ with $k^* = \inf\{ j \in \NN: \Delta_{j+1} = \infty \}$ and $\rho_k = \sum_{j=1}^k \Delta_j$. This means that we can estimate $k^*$ to be stochastically dominated by a geometric random variable with success parameter $\lambda$. Therefore using triangle inequality in the first inequality, Cauchy-Schwarz in the second one and Proposition \ref{prop:tails_of_waiting_time} in the last one we obtain:

\begin{equation}\label{eq:geometric_series_tails}
\begin{aligned}
\EE \left[ \exp\left( b \psi_{\alpha,\kappa}^{\kappa} \right) \right] \leq & C \sum_{j \geq 1} \EE\left[ \exp\left( b \rho_{k^*}^{\kappa} \right) \indic{ k^* = j } \right] \\
\leq & C \sum_{j\geq 1} \left( \EE \left[ \exp\left( 2b \rho_j \right) \indic{ \rho_j < \infty} \right] \right)^{1/2} \PP\left( k^* = j \right)^{1/2} \\
\leq & C \sum_{j \geq 1} (1+2bM)^{k/2} (1-\lambda)^{k/2},
\end{aligned}
\end{equation}

where the constant $C$ depends on $t^*, b, \lambda$. The series in \eqref{eq:geometric_series_tails} is bounded by a geometric one with the parameter $q = (1+2bM)^{1/2}(1-\lambda)^{1/2}$, which implies that one can always find small enough $b := b(\lambda)$ such that the series converges.
\end{proof}

Now we can proceed with the proper proof.

\begin{proof}[Proof of Theorem \ref{thm:main}]

 By Lemma \ref{lem:curve_gam_neg} or \ref{lem:curve_gam_pos} there exists $\alpha \in (H,\inv{1-\gamma}]$ and some constant $C > 0$ (which can change line by line) such that:

\[ \forall u > 0\;X^1_{\psi_{\alpha,\kappa}+u} > \varphi(M, u) -  C q^{k(u) \cdot \alpha}  \quad  { k(u) = 1 + \bar{k}(u) \quad \bar{k}(u) := \sup\{ k \in \NN: T_k < u \},  } \]

where $T_k$ is defined in \eqref{eq:geometric_times}. Set $\nu = \inv{1-\gamma} - \alpha$ and by scaling  (Proposition \ref{prop:scaling}) this means:

\begin{equation}\label{eq:lower_bound_eps_scale}
X^{\eps}_{t_{\eps} \psi_{\alpha,\kappa} + t_{\eps} u} > x_{\eps} \varphi(M,u ) \left( 1 - c_w q^{-k(u) \nu} \right).
\end{equation}

It then holds using \eqref{eq:transpoint} \[ x_{\eps} \varphi(M ,u) = x_{\eps} C_E \left( \left( C_E^{-1} M \right)^{1-\gamma} + u \right)^{\inv{1-\gamma}}= \varphi\left( M x_{\eps}, u t_{\eps} \right) .\]

We then set $t = t_{\eps} u$ - we will now estimate $q^{-k(u)}$ as a function of $t$. Moreover, let us set $\psi_{\eps,\alpha,\kappa} = t_{\eps} \psi_{\alpha,\kappa}$. 
Using the fact that $T_k$ is a geometric series one has:

\begin{equation}\label{eq:geom_bound}
q^{\bar{k}(u)}\geq C_q  (u \lor 1)
\end{equation}

for some $C_q>0$, where $\lor$ came from the uniform lower bound on $\bar{k}$. As a result we get \[ q^{-\bar{k}(u) \nu} \leq c_{\nu,q} \left( t^{-\nu} t_{\eps}^{\nu} \land 1 \right). \] Therefore the bound for all $t$ in \eqref{eq:lower_bound_eps_scale} can be rewritten as:

\begin{align*}
X^{\eps}_{\psi_{\eps,\alpha,\kappa} + t} & > \varphi( x_{\eps} M, A^+ t) -  x_{\eps} (u \lor 1)^{\alpha} \\
& > \varphi( M x_{\eps}, A^+ t ) - x_{\eps} ( t_{\eps}^{-1} t \lor 1 )^{\alpha} .
\end{align*}

 Clearly, for $t < t_{\eps}$ the second term is bounded by constant - by the choice of $c_w$ made in previous sections (Lemma \ref{lem:curve_gam_neg} and \ref{lem:curve_gam_pos}) we know that $X^{\eps}$ will stay positive. For $t > t_{\eps}$ we use Proposition \ref{prop:scaling} again and obtain:

\[ X^{\eps}_{\psi_{\eps,\alpha,\kappa} + t} > \varphi( M x_{\eps}, A^+ t ) - \eps t_{\eps}^{H-\alpha} t^{\alpha}, \]

as stated in the Theorem. Note that the constant in front of the error term is always smaller than one - this is because the additional multiplicative constant coming from \eqref{eq:geom_bound} is not bigger than $q^{\alpha}$. On the level of $\eps = 1$ we have then the bound $\absv{ w_t}  \leq c_w q^{\alpha} t^{\alpha}$. We picked $c_w \leq 3K_A$, where $K_A$ is chosen to be a constant smaller than $q^{-\alpha}/3$ (see Definition \ref{def:stage2_constants}). Therefore $c_w q^{\alpha} \leq 1$ as desired.  Obviously, the same computations follow if $X_{\psi_{\eps,\alpha,\kappa}} < 0$.

To see that $\PP\left( C^+ \right) >0$, we simply bound it from below by the probability of succeeding at the first attempt - for instance for some suitable constants $c_1,c_2,c_3$ where $c_3 < c_1$ (in fact $c_3$ should be taken much smaller than $c_1$) we take the following:

\[
\PP\left( C^+ \right) \geq \PP\left( \max(L(-\infty, -1), S(-1,0)) < c, Z \in (c_1, c_2), \norm{ b }_{\ccc{1/2-\delta};[0,1]} < c_3, A_w, A_g, k_1^f = \infty \right),
\]

where $Z$ and $\norm{ b }_{\ccc{1/2-\delta};[0,1]}$ are the Brownian decomposition for noise (see \eqref{eq:bm_bridge}) in $t \in [0,\Delta_0]$ are such that $X_{\Delta_0} > 0$. The constants $c_2, c_3$ are small enough such that after the admissible condition holds. Then two events $A_g,A_w$ (Definition \ref{def:getting_out_event} and Lemma \ref{lem:getting_out_proba}) hold and finally we manage to sample the infinite sequence of noises in Lemma \ref{lem:infinite_brownian_holder_norms}. In the case of $C^-$ we can change the sign of $c_1,c_2$.
\end{proof}

We conclude with the (easy) proof of the upper bound 
\begin{proof}[Proof of Proposition \ref{prop:upper_bound}]
For $\gamma > 0$ we let, for a  given $t \geq 0$, 
\[
\tau = \sup \{ s \leq t : X^\eps_s =0 \}
\]
 and we can then directly apply Lemma \ref{lem:positive_comparison} on $[\tau,t]$ with $\bar{w} = \eps  \sup_{\tau \leq s \leq t} | W^H_s - W^H_{\tau} | \leq 2 \eps \norm{ \wh }_{\infty;[0,1]}$, which gives

\[ X^{\eps}_t \leq \varphi(  { 2 \bar{w} }, A^+ t), \]

and similarly for the lower bound. 

For $\gamma < 0$ we define

\[ \tau = \sup\{ u < T\; X_t \leq \varphi(0, A^+ t) \}. \]

Then:

\[ \forall t \geq \tau, \;X_t - \varphi(0, A^+ t) = \int_{\tau}^t  { f(X_r) } dr - A^+ \int_{\tau}^t \absv{ \varphi(0, A^+ r) }^{\gamma} dr + \eps (w_t - w_{\tau}).\]

By the definition of $\tau$ and monotonicity of $f$ on $(0,+\infty)$, we have $ { f(X_r) } < A^+  \absv{ \varphi(0, A^+ r) }^{\gamma}$ for $r \in [\tau, t]$, therefore the right hand-side is bounded with $2 \eps \norm{ w }_{\infty;[0,T]}$.
\end{proof}

\section{Some technical estimates}\label{sec:technical}

\begin{lemma}\label{lem:subgaussian_scaling}
Fix $\delta > 0$ and define for $0<a<b$ :

\[ M(a,b) = \inv{\absv{b-a}^{\delta}} \sup_{a<u<v<b} \frac{\absv{B_v - B_u}}{ \absv{v-u}^{1/2-\delta}}. \] 

Then there exists $\lambda > 0$ such that for all $a,b$, $\EE \left[ \exp\left( \lambda M(a,b)^2 \right) \right] < 2$.
\end{lemma}

\begin{proof}
Set $\bar{B}_v = B_v - B_a =_{\PP} B_{v-a}$. Then using Brownian scaling:

\[ B_v - B_u =_{\PP} \absv{b-a}^{1/2} \left( B_{\frac{v-a}{b-a}} - B_{\frac{u-a}{b-a}} \right) \]

Similarly we multiply $v-u$ by one and get:

\[ \frac{\absv{B_v - B_u}}{ \absv{v-u}^{1/2-\delta}} =_{\PP} \absv{b-a}^{\delta} \frac{ B_{\frac{v-a}{b-a}} - B_{\frac{u-a}{b-a}} }{ \absv{ \frac{v-a}{b-a} - \frac{u-a}{b-a} }^{1/2-\delta} } =_{\PP} \absv{b-a}^{\delta} M(0,1). \]

Taking the supremum over $u,v$ and dividing by $\absv{b-a}^{\delta}$ we get the claim by Gaussian tails of H\"older norm of $B$ on unit interval.
\end{proof}

We give here the proof of Lemma \ref{lem:bridge_in_the_past}.

\begin{proof}
We integrate by parts with taking $dB_r = d(B_r - B_1)$. Note that in this case the term standing by the $Z$ part of decomposition would be negative, therefore we bound it in the absolute terms. Since in this regime we may need to take $z$ of sublinear growth, let us fix $\beta \in (0,1)$ and proceed in an identical way:
\begin{equation}\label{eq:bound_h_small_zw}
 {
\begin{aligned}
& \absv{ \int_0^1 \left( 1 + w + z - r \right)^{H-1/2} - \left( 1 +w - r \right)^{H-1/2} dB_r } \\ = & \absv{ B_1 - B_0} \left( (1+w+z)^{H-1/2} - (1+w)^{H-1/2} \right) \\ & \;\;\; + \absv{ (H-1/2) \int_0^1 \left( \left( 1 + w + z - r \right)^{H-\frac{3}{2}} - \left( 1+w-r \right)^{H-\frac{3}{2}} \right) (B_r - B_1) dr } \\ 
< &  z^{\beta} Z w^{H-1/2-\beta} +  Z z^{\beta} w^{H-\frac{3}{2} - \beta}  + z^{\beta} \norm{ b }_{\ccc{1/2-\delta};[0,1]} w^{H-\frac{3}{2} - \beta} .
\end{aligned}
}
\end{equation}

Clearly the second term decays faster with $w$ and the desired bound follows.
\end{proof}

We can now give a proof of Lemma \ref{lem:new_noise_bridge_decomp}. 

\begin{proof}

We use integration by parts for $\wh$ - \eqref{eq:ibp} - and \eqref{eq:bm_bridge} to get:

\begin{align*}
\int_0^s (s-r)^{H-1/2} dB_r = & s^{H-1/2} (B_s-B_0) + (H-1/2) \int_0^s (s-r)^{H-3/2} (B_r - B_s) dr \\
= &  { s^{H+1/2} Z \inv{H+1/2}} \\ & \;\;\;+ s^{H-1/2} b_{s} +  {(H- 1/2)} \int_0^s (s-r)^{H-3/2} \left( b_{r} - b_{s} \right) dr.
\end{align*}

We denote again the last two terms with $b^H_s$ and using H\"older estimates get an analogous bound, but with the constant $C_{H,w} = 1 + \frac{1/2-H}{ H-\delta}$. To show the bound on the H\"older norm of $b^H$, we recall Theorem 2.8 in \cite{Pic11} - using $I^{\alpha}$ for the fractional integration operator introduced there, we observe that $b^H = I^{H-1/2} b$ and then apply its first statement with $\alpha = H-1/2, \beta = 1/2-\delta, \gamma = 0$, in which case $\mathbb{H}^{\beta,0} = \ccc{\beta}$. 

\vspace{0.15cm}
For the second claim we use the fact that for all $r \in [0,1]$ a map $r \mapsto (1+z-r)^{H-1/2}$ is Lipschitz for any $z>0$, so that one can interpret the integral in a Young sense. We insert \eqref{eq:bm_bridge} and in the second term we use integration by parts for Young integrals.

\begin{align*}
\int_0^1 (1+z -r )^{H-1/2} dB_r = & Z \int_0^1 (1+z-r)^{H-1/2} dr + (H-1/2) \int_0^1 (1+z-r)^{H-1/2} db_{r} \\
= & Z h(z) + (H-1/2) \int_0^1 (1+z-r)^{H-3/2} b_{r} dr ,
\end{align*}

where $h(z)$ has stated properties by the monotonicity of the integral. The second term is bounded using H\"older continuity of $f$ - one gets uniform bound by observing that the integrand is decreasing in $z$, the value of the constant $C_H$ comes from $\int_0^1 r^{H-1} dr \leq \inv{H}$.
\end{proof}

We recap here a proof already given in the introduction to Section 4.2 in \cite{PR20}.
 
\begin{proposition}\label{prop:gaussian_tails_norms}
Let $U$ be a random time adapted to filtration $(\FF_t)_{t \geq 0}$ and let $T \in \FF_U$. Let $L(U, T), S(T-1,T), T > U + 1$ be the norms from Definition \ref{def:long_short}. Then there exists $\lambda_L$ (respectively $\lambda_S$) such that \[ \EE^{\FF_{U}} \left[\exp\left( \lambda_L L(U, T)^2 \right) \right] \leq 2 \quad \EE^{\FF_U} \left[ \exp\left( \lambda_S S(T-1,T)^2 \right) \right] \leq 2.\]
\end{proposition}

\begin{proof}
For any deterministic $s<t$ there holds $(B_r - B_t)_{r \in [s,t]} =_{\PP} (B_r)_{r \in [0,t-s]}$. Therefore conditionally on $\FF_U$ the random variable $L(U,T)$ is equal in law to the following norm:

\[ V(T-U) = \sup_{r \in [0, T-U]} \frac{ \bar{B}_r }{ (1+r)^{1/2+\delta} }, \]

 {where $\bar{B}$ is another Brownian motion } independent from $\FF_U$. This norm is obviously stochastically dominated by the one with $T - U$ replaced with $\infty$, which has Gaussian tails (see for instance \cite{Pic11}, Section 2). The constant $\lambda_L$ can be picked as the one for which  { $\EE\left[ \exp\left( \lambda_L V(\infty)^2 \right) \right] \leq 2$ }holds.

The same argument applies to $S(T-1,T)$, except that we then compare with $\sup_{r \in [0,1]} r^{-1/2+\delta} \absv{B_r}$ and will not use domination by an equivalent norm on the full real line.
\end{proof}

Here we prove that the combination of constants needed for Stage 2 indeed exists.

\begin{proof}[Proof of  { Proposition \ref{prop:stage2_constants} }]

By the first constraint \eqref{eq:ca_tstar} one has $c_{A,f} < (1+t^*)^{H-1} \leq 2\left( U t_e^{-1/2-\delta} \right)^{\vartheta(H-1)}$. First let us check the possible values of $\vartheta$ - the constraint \eqref{eq:vartheta_lower} is going to be satisfied if $\vartheta < \left( 1/2 + \delta \right)^{-1}$. 

We now find $t_e, \delta$ such that \eqref{eq:girsanov_proba} can hold. We fix $\delta < 1/2+H(\gamma-1)$ - such $\delta$ always exists by $\gamma > 1 - \inv{2H}$. Setting $\delta ' = 1/2 + H(\gamma-1)$ one takes $t_e$ such that $t_e^{-2(\delta-\delta')} < 9$, so this will be satisfied for any $t_e <1$. At this stage we do not specify the choice of $\delta$.

We can work with $K_A = 1$, the proof of general case follows in the identical way. We then pass to \eqref{eq:teu}. For $H < 1/2$ the bounds \eqref{eq:teu}, \eqref{eq:vartheta_lower} impose both upper and lower bounds on $t_e, U$ (one can replace $t_e^H$ with 1):

\begin{equation}\label{eq:three_one_thirds}
t_e^{-\vartheta (H+\delta) (1/2+\delta)} U^{\vartheta(H-1/2)} \geq 5^{\inv{\alpha}} \quad t_e^{-3\delta} U^{-\vartheta(1/2-3\delta)} \leq 1/3 \quad (t^*)^{1+\vartheta(H-1/2-\beta)} < 1/3.
\end{equation}

By the choice of $\delta$ in \eqref{eq:ca_tstar} the second inequality can be replaced with $U^{-\vartheta/2} \leq 2/3$ - note the absence of the lower bound on $\delta$, so it can be always decreased if need be. We take any $\beta > \inv{\vartheta} + H - 1/2 > H$ (this follows from $\vartheta < 2$) and then the power in the last inequality in \eqref{eq:three_one_thirds} is negative. For fixed $U$ (for instance $U = \left( \frac{3}{2} \right)^{2/\vartheta}$) therefore it suffices to increase $t_e^{-1}$ in order to get the last inequality. Note that for $H < 1/2$ the power on $U$ in the first inequality is negative - to obtain the desired inequality we increase $t_e^{-1}$ even further. One can take then, for instance:

\[ U = \left( \frac{3}{2} \right)^{2/\vartheta} \quad t_e^{-1} = \left( 5^{\inv{\alpha}} \left( \frac{3}{2} \right)^{1-2H} \right)^{\inv{\vartheta(H+\delta)(1/2+\delta)}} \lor 3^{\inv{\vartheta(H-1/2-\beta)+1}} . \]

Clearly there exists $\delta(H,\vartheta)$ such that the second inequality in \eqref{eq:ca_tstar} also holds.

The bound on $C_W$ is seen to be easily satisfied in the absence of the lower bound. 
\end{proof}
%
%

\section*{Acknowledgements}
The authors are grateful to two anonymous referees for their careful reading and remarks which helped to improve the clarity of the presentation. This work was done when both authors were employed at Université Paris-Dauphine.

%


\bibliographystyle{alpha} 
\bibliography{refs}       


\bigskip
\bigskip

\noindent ($^{1}$) 
LAMA, Université Gustave Eiffel \\
Cité Descartes, Bâtiment Copernic \\
5 boulevard Descartes \\
77454 Marne-la-Vallée cedex 2 \\
email: paul.gassiat@univ-eiffel.fr


\noindent ($^{2}$) LMBA, \\
Université de Bretagne Occidentale, \\
6 avenue Le Gorgeu, CS 93837, \\
29238 Brest cedex 3 \\
email: lmadry@univ-brest.fr

\end{document}